\documentclass[journal,twoside,web]{ieeecolor}
\usepackage{generic}
\usepackage{cite}
\usepackage{amsmath,amssymb,amsfonts}
\usepackage{algorithmic}
\usepackage{graphicx}
\usepackage{epstopdf}
\usepackage{algorithm,algorithmic}
\usepackage{hyperref}
\hypersetup{hidelinks=true}
\usepackage{textcomp}
\usepackage{booktabs}
\usepackage[none]{hyphenat}
\usepackage{changes} %%\added  \deleted
\newtheorem{definition}{Definition}
\newtheorem{lemma}{Lemma}
\newtheorem{theorem}{Theorem}
\newtheorem{corollary}{Corollary}
\newtheorem{example}{Example}
\newtheorem{proposition}{Proposition}
\newtheorem{remark}{Remark}
\newtheorem{assumption}{Assumption}
\DeclareMathOperator*{\esssup}{ess\,sup}

\def\BibTeX{{\rm B\kern-.05em{\sc i\kern-.025em b}\kern-.08em
    T\kern-.1667em\lower.7ex\hbox{E}\kern-.125emX}}
\begin{document}
\title{Detectability, Riccati Equations, and the Game-Based Control of Discrete-Time MJLSs with the Markov Chain on a Borel Space
%(Extended Version)
}
\author{Chunjie Xiao, Ting Hou, Weihai Zhang, \IEEEmembership{Senior Member, IEEE}, and Feiqi Deng, \IEEEmembership{Senior Member, IEEE}
\thanks{Chunjie Xiao and Ting Hou (corresponding author) are with the School of Mathematics and Statistics, Shandong Normal University, Jinan 250014, China (e-mail: xiaocj6\_sd@163.com; ht\_math@sina.com).}
\thanks{Weihai Zhang is with the College of Electrical Engineering and Automation, Shandong University of Science and Technology, Qingdao 266590, China
(e-mail: w\_hzhang@163.com).}
\thanks{Feiqi Deng is with the School of Automation Science and Engineering, South China University of Technology, Guangzhou 510640, China (e-mail: aufqdeng@scut.edu.cn).}}

\maketitle

\begin{abstract}
   In this paper, detectability is first put forward for discrete-time Markov jump linear systems
   with the Markov chain on a Borel space ($\Theta$, $\mathcal{B}(\Theta)$).
   Under the assumption that the unforced system is detectable,
   a stability criterion is established relying on the existence of the positive semi-definite solution to the generalized Lyapunov equation.
   It plays a key role in seeking the conditions that guarantee the existence and uniqueness
   of the maximal solution and the stabilizing solution for a class of general coupled algebraic Riccati equations (coupled-AREs).
   Then the nonzero-sum game-based control problem is tackled,
   and Nash equilibrium strategies are achieved by solving four integral coupled-AREs.
   As an application of the Nash game approach, the infinite horizon mixed $H_{2}/H_{\infty}$ control problem is studied,
   along with its solvability conditions.
   These works unify and generalize those set up in the case
   where the state space of the Markov chain is restricted to a finite or countably infinite set.
   Finally, some examples are included to validate the developed results, involving a practical example of the solar thermal receiver.
\sloppy{}
\end{abstract}

\begin{IEEEkeywords}
Borel space, detectability, Markov jump systems, Riccati equations, the game-based control.
\end{IEEEkeywords}

\section{Introduction}
\label{sec:introduction}
\IEEEPARstart{S}{ome} dynamic systems, due to their innate vulnerability,
frequently undergo abrupt structural changes triggered by sudden phenomena
like random environmental disturbances, component or interconnection failures, or repairs.
Markov jump linear systems (MJLSs) have actually been shown to play an appropriate role in modeling this kind of dynamic systems.
More formally, since it consists of a set of subsystems,
MJLS enables the simulation of different scenarios that may occur during the dynamical evolution.
And, the stochastic switching governed by the Markov chain can precisely describe the random behavior exhibited by the simulated system.
Taking a solar thermal receiver as an example,
it was modeled as a discrete-time MJLS in \cite{BookCosta2005}
by using a two-state Markov chain to simulate sunny and cloudy weather conditions.
Over the past few decades, MJLSs have been extensively studied
\cite{EFCosta2005SIAM, Aberkane2015SIAM,  Aberkane2020TAC, Dragan2022TAC}
and many significant results on analysis and controller design have been yielded,
which can be applied in the fields
such as economics, energy \cite{Vargas2013}, communication \cite{Ouyang2021},
and robotic manipulation systems \cite{Siqueira2009}.

Nash game was originally proposed to address issues in economics,
specifically for analyzing strategic interactions and decision-making processes among multiple players.
In 1994, the two-player Nash game approach was first applied to
solve the mixed $H_{2}/H_{\infty}$ control problem for deterministic systems, see \cite{Limebeer1994h2hinfNash}.
Subsequently, related studies were carried out to explore the finite- or infinite horizon $H_{2}/H_{\infty}$ control problems
within the MJLSs framework,
including the continuous-time case \cite{HuangZhang2008}, and the discrete-time case \cite{Hou2013JGO}.
In these works, Nash equilibrium strategies were expressed in the form of linear state feedbacks,
with gain matrices obtained by solving coupled Riccati equations.
As for the difference between the finite- and infinite horizon scenarios,
the former focuses solely on optimizing the performance indices,
whereas the latter also requires to ensure the stability of the system dynamics.
That is, the solution to coupled algebraic Riccati equations (coupled-AREs) in the infinite horizon scenario
needs to be a stabilizing one.
It was generally recognized that detectability might serve as an important role
in guaranteeing the existence of the stabilizing solution to coupled-AREs \cite{BookCosta2005, BookDragan2010}.

Notice that the aforementioned works were largely centred on MJLSs with the Markov chain taking values in a finite or countably infinite set.
As a matter of fact, the differences of the state space may cause different properties of MJLSs.
In recent years, efforts have been made to explore MJLSs with the Markov chain on a general
(not necessarily countable) state space such as a Borel space.
References such as \cite{Meynsp2009, Li2012, Costa2014, Costa2015, Costa2016, Costa2017} can be consulted.
Among these works, \cite{Meynsp2009}, \cite{Li2012}, and \cite{Costa2014} conducted the stability analysis,
\cite{Costa2015} and \cite{Costa2016} treated some optimal control problems,
and the filtering design was undertaken in \cite{Costa2017}.
It should be stressed that this kind of MJLSs have great potential for practical applications.
For instance, in \cite{Masashi2018}, a Markov chain on a Borel space was used to
simulate continuous-valued random delays that existed in communication.
Take the solar thermal receiver as another example.
In Example \ref{1468Solarthermal} of this paper,
the solar thermal receiver is modeled by a discrete-time MJLS $x(k+1)=a(\vartheta(k))x(k)$
with the Markov chain $\{\vartheta(k), k\in\mathbb{N}\}$ on a Borel space.
The strength of $\vartheta(k)=(i,t)$, $k\in\mathbb{N}$ is that,
it not only considers the weather conditions of sunny and cloudy (sunny represented by $i=1$; cloudy represented by $i=2$),
but also provides a precise description of the instantaneous solar radiation effect
on the system parameters under the sunny/cloudy weather condition (represented by the variable $t\in [0,1]$).

In this paper, we place our attention on detectability and the game-based control
of discrete-time MJLSs with the Markov chain on a Borel space.
The first thing to notice is that the Markov chain on this general space,
its state space is no longer restricted to a discrete set,
and its statistical characteristics are determined by the initial distribution and the stochastic kernel.
Therefore, it differs from the countable framework by using integral operators rather than the summation operations used previously.
This brings technical difficulties.
Furthermore, the results obtained in this paper heavily rely on solving equations such as the generalized Lyapunov equations and coupled-AREs.
Significant challenges include ensuring boundedness (which is apparent in the finite-dimensional case, see \cite{BookDragan2010}),
measurability (which is immediate in the countable case),
and integrability (which corresponds to summability in the countable case) of the solutions to the associated equations.
Thus, system analysis and the game-based controller design completed in this paper are not trivial.

The main contributions of this paper include the following:
\begin{itemize}
\item
      %Detectability is first defined as exponential mean-square stability of the estimation error,
%      which is induced by the state estimator.
      It has been shown in Theorem \ref{Mthe1495if} that under the detectability hypothesis, the existence of a positive semi-definite solution to the generalized Lyapunov equation ensures the autonomous system being exponentially mean-square stable (EMSS).
      This result can be viewed as a generalization of the classical Barbashin-Krasovskii type stability criterion
      (see \cite{BookDragan2014} for continuous-time finite MJLSs), while the observability hypothesis has been weakened to detectability.
      Actually, the condition proposed in Theorem \ref{Mthe1495if} is not only sufficient but also necessary for EMSS.
      The proof heavily relies on the equivalence between EMSS and exponential mean-square stability with conditioning (EMSS-C)
      established in Theorem \ref{emss-emss-c}.

\item A class of general algebraic Riccati equations coupled via an integral has been dealt with,
      which is closely connected with several control problems like the standard LQ optimization and the $H_{\infty}$ control.
      The existence conditions for the maximal/stabilizing solution are presented.
      Moreover, it has been demonstrated that stabilizability and detectability can guarantee that the solution to coupled-AREs
      associated with the standard LQ optimization is unique and stabilizing.
      In establishing these results, the stability criterion proposed in Theorem \ref{Mthe1495if} plays a crucial part.

\item In Theorem \ref{gameth}, under the assumption that the considered system is detectable,
      a two-player nonzero-sum game has been settled and Nash equilibrium strategies have been characterized
      through the stabilizing solutions of a set of coupled-AREs.
      As a by-product, the infinite horizon $H_{2}/H_{\infty}$ control problem has been addressed
      jointly by Theorem \ref{gameth} and the bounded real lemma (BRL) built in \cite{Xiao2023}.
      In Example \ref{1468comparison}, the superiority of the mixed $H_{2}/H_{\infty}$ control is exhibited by compared with
      the $H_{\infty}$ control.
\end{itemize}

The remainder of this paper is organized as follows:
Section \ref{S2:Preliminaties} includes two parts.
Specifically, Subsection \ref{sub1} contains some preliminaries, and Subsection \ref{subsec2} is a review of the results on exponential stability.
The definition of detectability is given in Section \ref{Detectability}, followed by
a Barbashin-Krasovskii type stability criterion under the detectability hypothesis.
Discussions on the existence of the maximal/stabilizing solution to coupled-AREs are involved in Section \ref{maxstaGCARE}.
Section \ref{Gamecontrol} addresses the infinite horizon controller design problems thoroughly,
and some illustrative examples are provided in Section \ref{anexample}.
Section \ref{Conclusions} concludes the paper with a summary.

\section{Preliminaries, the Review of Stability, and the Decomposition of Matrix-Valued Functions}\label{S2:Preliminaties}
\subsection{Preliminaries}\label{sub1}
\emph{Notations}:
$\mathbb{N}^{+}:=\{1,2,\cdots\}$; $\mathbb{N}:=\{0\}\bigcup\mathbb{N}^{+}$.
For integers $k_{1}\leq k_{2}$, $\overline{k_{1},k_2}:=\{k_{1},k_{1}+1,k_{1}+2,\cdots,k_{2}\}$.
$\emptyset$ stands for the null set.
$\mathbb{R}^{n}$ denotes the $n$-dimensional real Euclidean space.
$\|\cdot\|$ is the Euclidean norm or operator norm.
For Banach spaces $\mathcal{H}$ and $\bar{\mathcal{H}}$, $\mathbf{B}(\mathcal{H},\bar{\mathcal{H}})$
represents the Banach space of all bounded linear operators from $\mathcal{H}$ to $\bar{\mathcal{H}}$,
and the norm of $\mathcal{L}\in\mathbf{B}(\mathcal{H},\bar{\mathcal{H}})$
is defined as $\|\mathcal{L}\|:=\sup_{x\in\mathcal{H},\ \|x\|=1}\{{\|\mathcal{L}x\|}$\}.
For brevity, $\mathbf{B}(\mathcal{H}):=\mathbf{B}(\mathcal{H},\mathcal{H})$.
For $\mathcal{L} \in \mathbf{B}(\mathcal{H})$, the spectral radius of $\mathcal{L}$ is defined as
$r_{\sigma}(\mathcal{L}):=\sup_{\lambda\in{\sigma(\mathcal{L})}}|\lambda|$,
where $\sigma(\mathcal{L})$ is the spectrum of $\mathcal{L}$.
A convex cone $\mathcal{H}^{+}\subseteq\mathcal{H}$ induces an ordering ``$\leq$'' on $\mathcal{H}$
by $U\leq \bar{U}$ if and only if (iff) $\bar{U}-U\in \mathcal{H}^{+}$, where $\bar{U},\ U\in\mathcal{H}$.
$\mathcal{L} \in \mathbf{B}(\mathcal{H})$ is called a positive operator
and denoted by $\mathcal{L}\geq 0$,
if $\mathcal{L}\mathcal{H}^{+}\subseteq\mathcal{H}^{+}$.
In particular, $\mathbb{R}^{n\times m}:=\mathbf{B}(\mathbb{R}^{m},\mathbb{R}^{n})$
represents the space composed of $n \times m$ dimensional real matrices with the norm $\|M\|:=[\lambda_{max}(M^{T}M)]^{\frac{1}{2}}$,
where $(\cdot)^{T}$ and $\lambda_{max}(\cdot)$ are the transpose and the maximum eigenvalue of a matrix.
$I$ represents the identity matrix with appropriate dimension.
$\mathbb{S}^{n}:=\{M\in{\mathbb{R}^{n\times n}}|M^{T}=M\}$;
$\mathbb{S}^{n+}:=\{M\in{\mathbb{S}^{n}}|M\geq 0\}$.
$\mathbf{E}\{\cdot\}$ stands for the mathematical expectation.
Throughout the paper, $n$ is a fixed positive integer representing the dimension of the system state under consideration.

In this paper, $\Theta$ is assumed to be a Borel subset of a Polish space
(i.e., a separable and complete metric space).
The Borel space $(\Theta, \mathcal{B}(\Theta))$ is defined as $\Theta$,
together with its Borel $\sigma$-algebra $\mathcal{B}(\Theta)$.
$\mu$ is a $\sigma$-finite measure on $\mathcal{B}(\Theta)$.

On a probability space $(\Omega, \mathfrak{F}, \mathbb{P})$,
define a Markov chain $\{\vartheta(k), k\in\mathbb{N}\}$ taking values in $\Theta$
with the initial distribution given by a probability measure $\mu_{0}$
and the stochastic kernel $\mathbb{G}(\cdot,\cdot)$ satisfying
$\mathbb{G}(\vartheta(k),\Lambda):=\mathbb{P}(\vartheta(k+1)\in \Lambda|\vartheta(k))$
almost surely $(a.s.),$  $ \forall k\in\mathbb{N}$ and $\Lambda\in\mathcal{B}(\Theta)$.
Suppose that for any $\ell\in\Theta$,
$\mathbb{G}(\ell,\cdot)$ has a density $g(\ell,\cdot)$ with respect to (w.r.t.) $\mu$,
that is, $\mathbb{G}(\ell,\Lambda)=\int_{\Lambda}g(\ell,t)\mu(dt)$ for any $\Lambda\in\mathcal{B}(\Theta).$
Note that the assumption on the Markov chain is reasonable when $\Theta$ is Borel.
For further details, please refer to \cite{Xiao2023}.

Let $\mathcal{H}^{n\times m}$ ($\mathcal{SH}^{n}$) denote the space of measurable matrix-valued functions
$Q(\cdot): \Theta\rightarrow \mathbb{R}^{n\times m}$ ($Q(\cdot): \Theta\rightarrow \mathbb{S}^{n}$).
$\mathcal{H}^{n\times m}_{1}=\big\{Q\in\mathcal{H}^{n\times m}\big| \|Q\|_{1}:=\int_{\Theta}\|Q(\ell)\|\mu(d\ell)<\infty\big\}$;
$\mathcal{H}^{n\times m}_{\infty}=\big\{Q\in\mathcal{H}^{n\times m}\big| \|Q\|_{\infty}:=\esssup\{\|Q(\ell)\|, \ell\in{\Theta}\}<\infty\big\}$.
$(\mathcal{H}_{1}^{n\times m},\|\cdot\|_{1})$ and $(\mathcal{H}^{n\times m}_{\infty},\|\cdot\|_{\infty})$ are Banach spaces (see \cite{Costa2014}).
$\mathcal{SH}^{n}_{1}=\big\{Q\in\mathcal{SH}^{n}\big| \|Q\|_{1}<\infty\big\}$;
$\mathcal{SH}^{n}_{\infty}=\big\{Q\in\mathcal{SH}^{n}\big| \|Q\|_{\infty}<\infty\big\}$.
$(\mathcal{SH}_{1}^{n},\|\cdot\|_{1})$ is an ordered Banach space
with the order induced by the convex cone $\mathcal{H}_{1}^{n+}=\big\{Q \in \mathcal{SH}_{1}^{n}| Q(\ell)\in{\mathbb{S}^{n+}}\ \mu\text{-}a.e.\big\}$;
$(\mathcal{SH}^{n}_{\infty},\|\cdot\|_{\infty})$ is an ordered Banach space with the order induced by
$\mathcal{H}_{\infty}^{n+}=\big\{Q \in \mathcal{SH}_{\infty}^{n} | Q(\ell)\in{\mathbb{S}^{n+}}
\ \mu\text{-almost everywhere on $\Theta$ ($\mu$-$a.e.$)}\big\}$.
$\mathcal{H}_{\infty}^{n-}=\big\{Q\in \mathcal{SH}_{\infty}^{n} | Q(\ell)\leq 0\ \mu\text{-}a.e.\big\}$;
$\mathcal{H}_{\infty}^{n+*}=\big\{Q \in \mathcal{H}_{\infty}^{n+} |Q(\ell)\gg 0 \ \mu\text{-}a.e. \big\}$,
where $Q(\ell)\gg 0$ means that $Q(\ell) \geq \xi I$ for some $\xi>0$.
Given $\mathcal{I}=\{\mathcal{I}(\ell)\}_{\ell\in\Theta}$, where $\mathcal{I}(\ell)=I$,
one can easily check that $\mathcal{I}\in\mathcal{H}_{\infty}^{n+*}$ is the identity element in $\mathcal{SH}_{\infty}^{n}.$
In this paper, all properties of the Borel-measurable functions,
including inequalities and equations, should be understood in the sense of $\mu$-$a.e.$,
or for $\mu$-almost all $\ell\in\Theta$, unless otherwise specified.

Consider the following MJLSs:
\begin{equation} \label{system}
\left\{
\begin{array}{ll}
x(k+1)=A(\vartheta(k))x(k)+B(\vartheta(k))u(k) +F(\vartheta(k))v(k),\ \\
z(k)=\left[
       \begin{array}{c}
         C(\vartheta(k))x(k) \\
         D(\vartheta(k))u(k) \\
       \end{array}
     \right],\ \forall k\in\mathbb{N}
\end{array}
\right.
\end{equation}
with $x(0)=x_{0}\in\mathbb{R}^{n}$ and $\vartheta(0)=\vartheta_{0}$,
where $x_{0}$ is a deterministic vector in $\mathbb{R}^{n}$
and $\vartheta_{0}$ is a random variable.
$x(k)\in{\mathbb{R}^{n}}$, $u(k)\in{\mathbb{R}^{m}}$, and $z(k)\in{\mathbb{R}^{p+q}}$ are the system state, input, and output. $v(k)\in{\mathbb{R}^{r}}$ is the external disturbance input.
$\mathfrak{F}_{k}$ represents the $\sigma$-field generated by $\{ \vartheta_0, \vartheta(1), \cdots, \vartheta(k)\}$.

Throughout this paper, we make the following assumptions.
\begin{assumption}\label{Assumption1}
$(i)$~$A=\{A(\ell)\}_{\ell\in\Theta}\in\mathcal{H}^{n\times n}_{\infty}$,
$B=\{B(\ell)\}_{\ell\in\Theta}\in\mathcal{H}^{n\times m}_{\infty}$,
$C=\{C(\ell)\}_{\ell\in\Theta}\in\mathcal{H}^{p\times n}_{\infty}$,
$D=\{D(\ell)\}_{\ell\in\Theta}\in\mathcal{H}^{q\times m}_{\infty}$,
and $F=\{F(\ell)\}_{\ell\in\Theta}\in\mathcal{H}^{n\times r}_{\infty}$;

$(ii)$~The Markov chain $\{\vartheta(k),k\in\mathbb{N}\}$ is directly accessible,
i.e., the operation mode of the system is known at each $k\in\mathbb{N}$;

$(iii)$~$D(\ell)^{T}D(\ell)=I$ $\mu\text{-}a.e.$;

$(iv)$~The initial distribution $\mu_{0}$ is absolutely continuous w.r.t. $\mu$.
\end{assumption}

Considering $(iv)$ of Assumption \ref{Assumption1},
it follows from the Radon-Nikodym theorem (for example, see Theorem 32.2 in \cite{BookBillingsley1995})
that there exists $\nu_{0}\in\mathcal{H}^{1+}_{1}$ such that for any $\Lambda\in{\mathcal{B}(\Theta)}$, $\mu_{0}(\Lambda)=\int_{\Lambda}\nu_{0}(\ell)\mu(d\ell)$.

We continue this section with giving the following operators:
For any $Q\in\mathcal{H}^{n\times n}_{\infty}$, $V\in\mathcal{H}^{n\times n}_{1}$,
$U\in\mathcal{H}^{n\times n}_{\infty}$, and $\ell\in\Theta$, define
\begin{align*}
&\mathcal{L}_{Q}(V)(\ell):=\int_{\Theta}g(t,\ell)
Q(t)V(t)Q(t)^{T}\mu(dt),\\
&\mathcal{E}(U)(\ell):=\int_{\Theta}g(\ell,t)U(t)\mu(dt),\nonumber \\ &\mathcal{T}_{Q}(U)(\ell):=Q(\ell)^{T}\mathcal{E}(U)(\ell)Q(\ell),\nonumber
\end{align*}
and the bounded bilinear operator
$$\langle V; U\rangle:=\int_{\Theta}tr(V(t)^{T}U(t))\mu (dt),$$
where $tr(V(t)^{T}U(t))$ denotes the trace of $V(t)^{T}U(t)$.

The following properties regarding the defined operators are useful.
\begin{proposition}\label{adjoint}\cite{Costa2014}
For any $Q\in\mathcal{H}^{n\times n}_{\infty}$, $V\in\mathcal{H}^{n\times n}_{1}$, and $U\in{\mathcal{H}^{n\times n}_{\infty}}$,
the following hold:

$(i)$~$\mathcal{L}_{Q}\in\mathbf{B}(\mathcal{H}^{n\times n}_{1})$,
and $\mathcal{L}_{Q}$ is a positive operator on $\mathcal{SH}^{n}_{1}$;

$(ii)$~$\mathcal{T}_{Q}\in\mathbf{B}(\mathcal{H}^{n\times n}_{\infty})$,
and $\mathcal{T}_{Q}$ is a positive operator on $\mathcal{SH}^{n}_{\infty}$;

$(iii)$~$\langle \mathcal{L}_{Q}(V); U\rangle= \langle V; \mathcal{T}_{Q}(U)\rangle$.
\end{proposition}

\begin{proposition}\label{lemma468L}
$(i)$~For any $Q\in \mathcal{H}^{n\times n}_{\infty}$ and $V\in \mathcal{SH}_{1}^{n}$,
$\mathcal{L}_{-Q}(V)(\ell) =\mathcal{L}_{Q}(V)(\ell)$ $\mu$-$a.e.$;

$(ii)$~For any $Q_{1}\in \mathcal{H}^{n\times n}_{\infty},\ Q_{2}\in \mathcal{H}^{n\times n}_{\infty}$, and $\bar{V}\in \mathcal{H}_{1}^{n+}$, $\mathcal{L}_{Q_{1}+Q_{2}}(\bar{V})(\ell) \leq(1+\xi^{2})\mathcal{L}_{Q_{1}}(\bar{V})(\ell)+(1+\frac{1}{ \xi^{2}})
\mathcal{L}_{Q_{2}}(\bar{V})(\ell)$
$\mu$-$a.e.$, where $\xi \neq 0$.
\end{proposition}
\begin{proof}
The proof can be made through simple calculations and is then omitted.
\end{proof}

In the remainder of this paper, unless otherwise specified, for any $Q\in\mathcal{H}^{n\times n}_{\infty}$,
$\mathcal{L}_{Q}$ and $\mathcal{T}_{Q}$ are restricted on $\mathcal{SH}_{1}^{n}$ and $\mathcal{SH}_{\infty}^{n}$, respectively.

\subsection{The Review of Stability}\label{subsec2}
The following MJLS, denoted by $(A|\mathbb{G})$ for convenience, is considered in this subsection:
$$x(k+1)=A(\vartheta(k))x(k),\ k\in \mathbb{N}.$$
Define a function sequence $\{X(k), k\in\mathbb{N}\}$ associated with $(A|\mathbb{G})$ as follows:
\begin{equation}\label{BArxX}
X(k+1)(\ell)=\mathbf{E}\{x(k+1)x(k+1)^{T}
g(\vartheta(k), \ell)\},\ \ell\in\Theta,
\end{equation}
where $X(0)=X_{0}$ and $X_{0}(\ell)=x_{0}x_{0}^{T}\nu_{0}(\ell)$.
According to Proposition 5.1 of \cite{Costa2014}, for each $k\in \mathbb{N},$ $X(k)\in\mathcal{H}_{1}^{n+}$ and
\begin{equation}\label{88Xklre}
X(k+1)(\ell)=\mathcal{L}_{A}(X(k))(\ell)\ \mu\text{-}a.e..
\end{equation}
Further, it can be concluded inductively that
\begin{equation}\label{495xk=Lk}
X(k)(\ell)=\mathcal{L}_{A}^{k}(X_{0})(\ell)\ \mu\text{-}a.e..
\end{equation}
%The recurrence relation of $\{X(k), k\in\mathbb{N}\}$ is given in Proposition \ref{bDTLE} via using the operator $\mathcal{L}_{A}$. The proof of $(i)$ follows a similar mathematical induction method as in Proposition 5.1 of \cite{Costa2014}.
%\begin{proposition}\label{bDTLE}
%For each $k\in \mathbb{N},$
%
%$(i)$~$X(k)\in\mathcal{H}_{1}^{n+}$, and $X(k+1)(\ell)=\mathcal{L}_{A}(X(k))(\ell)\ \mu\text{-}a.e.$;
%
%$(ii)$~$X(k)(\ell)=\mathcal{L}_{A}^{k}(X_{0})(\ell)\ \mu\text{-}a.e..$
%\end{proposition}
%\begin{proof}
%$X(0)\in{\mathcal{H}_{1}^{n+}}$ is valid since  $\|X(0)\|_{1}=\int_{\Theta}\|x_{0}x_{0}^{T}\|\nu_{0}(\ell)\mu(d\ell)\leq\|x_{0}\|^{2}.$
%From \eqref{BArxX}, we have that
%\begin{equation*}
%\begin{split}
%X(1)(\ell)=\int_{\Theta}A(t)x_{0}x_{0}^{T}A(t)^{T}g(t,\ell)\nu_{0}(t)\mu(dt)=\mathcal{L}_{A}(X_{0})(\ell)
%\end{split}
%\end{equation*}
%$\mu\text{-}a.e..$
%Now it is assumed that $(i)$ holds for $k-1$.
%By $(i)$ of Proposition \ref{adjoint}, one can get that $X(k)\in\mathcal{H}_{1}^{n+}$.
%In view of \eqref{BArxX}, we obtain that
%$$
%\begin{aligned}
%&X(k+1)(\ell)\\
%=&\int_{\Theta} A(t)x(k)
%x(k)^{T}A(t)^{T}g(t,\ell) \nu_{k}(t) \mu(d t)\\
%=&\int_{\Theta} g(t,\ell)A(t)\!\int_{\!\Theta}x(k)
%x(k)^{T}\nu_{k-1}(s) g(s,t)\mu(d s) A(t)^{T} \mu(d t)\\
%=&\int_{\Theta} g(t,\ell)A(t)\mathbf{E}\{x(k)x(k)^{T}
%g(\vartheta(k-1), t)\} A(t)^{T} \mu(d t)\\
%=&\mathcal{L}_{A}(X(k))(\ell)\ \mu\text{-}a.e..
%\end{aligned}
%$$
%Therefore, $(i)$ is established by mathematical induction, and then $(ii)$ can be concluded inductively.
%\end{proof}

The following lemma explores the linkage between the system state $x(k)$ of $(A|\mathbb{G})$ and $\{X(k),\ k\in\mathbb{N}\}$.
\begin{lemma}\label{lemma345XxXxl2}
For any $Q\in \mathcal{H}^{n\times n}_{\infty}$ and $k\in \mathbb{N}$,
$$\|QX(k)Q^{T}\|_{1}\leq\mathbf{E}\{\|Q(\vartheta(k))x(k)\|^{2}\}.$$
Moreover, $\sum_{k=0}^{\infty}\mathbf{E}\{\|Q(\vartheta(k))x(k)\|^{2}\}<\infty$ iff $\sum_{k=0}^{\infty}\|QX(k)Q^{T}\|_{1}<\infty$.
\end{lemma}
\begin{proof}
For each $k\in\mathbb{N}$, we have that
\begin{equation}\label{P(k)x(k)}
\begin{aligned}
&\mathbf{E}\{\|Q(\vartheta(k))x(k)\|^{2}\}\\
=&tr\{\mathbf{E}\{Q(\vartheta(k))x(k)x(k)^{T}Q(\vartheta(k))^{T}\}\}\\
=&\int_{\Theta}tr\{Q(\ell)X(k)(\ell)Q(\ell)^{T}\} \mu(d\ell).
\end{aligned}
\end{equation}
From Proposition 1 in \cite{Xiao2023}, for each $k\in\mathbb{N}$, it follows that
\begin{equation*}
\begin{aligned}
\|Q(\ell)X(k)(\ell)Q(\ell)^{T}\|\leq &tr\{Q(\ell)X(k)(\ell)Q(\ell)^{T}\}\\
\leq&n \|Q(\ell)X(k)(\ell)Q(\ell)^{T}\|\ \mu\text{-}a.e..
\end{aligned}
\end{equation*}
Combining this with \eqref{P(k)x(k)} gives that
$$\|QX(k)Q^{T}\|_{1}\leq \mathbf{E}\{\|Q(\vartheta(k))x(k)\|^{2}\}\leq n \|QX(k)Q^{T}\|_{1},$$
and the equivalence between $\sum_{k=0}^{\infty}\mathbf{E}\{\|Q(\vartheta(k))x(k)\|^{2}\}<\infty$ and $\sum_{k=0}^{\infty}\|QX(k)Q^{T}\|_{1}<\infty$
is immediately worked out.
\end{proof}

Now we review two notions of exponential stability.
\begin{definition}\label{Defemssemssc}\cite{Xiao2023}
$(i)$~$(A|\mathbb{G})$ is said to be EMSS if there exist $\beta \geq 1$, $\alpha \in(0,1)$ such that for any initial conditions $(x_0,\vartheta_{0})$, we have
$\mathbf{E}\{\|x(k)\|^{2}\} \leq \beta \alpha^{k}\|x_{0}\|^{2},\ \forall k\in \mathbb{N}$;

$(ii)$~$(A|\mathbb{G})$ is said to be EMSS-C if there exist $\beta \geq 1$, $\alpha \in(0,1)$ such that
for any initial conditions $(x_{0}, \vartheta_{0})$, we have
$\mathbf{E}\{\|x(k)\|^{2}|\vartheta_{0}=\ell\} \leq \beta \alpha^{k}\|x_{0}\|^{2},\ \forall k\in \mathbb{N}$ holds for $\mu$-almost all $\ell\in\Theta_{\vartheta_0}.$ Here,
$\Theta_{\vartheta_0}=\{\ell\in\Theta|\nu_{0}(\ell)>0 \ \mu\text{-}a.e.\}.$
\end{definition}
\begin{remark}
$(A|\mathbb{G})$ being EMSS-C means that
the conditional expectation $\mathbf{E}\{\|x(k)\|^{2}|\vartheta_{0}=\ell\}$
converges exponentially to zero for $\mu$-almost all  $\ell\in\Theta_{\vartheta_0},$ independent of initial conditions $(x_{0},\vartheta_{0})$.
This EMSS-C concept generalizes the notion of
exponential stability in the mean square with conditioning of type I (ESMS-CI from \cite{BookDragan2010}) or uniform exponential stability in conditional mean (from \cite{Ungureanu2014Optimization})
to the case where the Markov chain takes values in a general Borel set.
\end{remark}

The equivalence between EMSS and EMSS-C is rigorously proved in the following theorem.
\begin{theorem}\label{emss-emss-c}
$(A|\mathbb{G})$ is EMSS iff $(A|\mathbb{G})$ is EMSS-C.
\end{theorem}
\begin{proof}
For the sufficiency part, refer to Proposition 5 in \cite{Xiao2023}.
We will prove the necessity by contradiction.
Following similar arguments as in the proof of Theorem 1 from \cite{Xiao2023},
we have that for
any initial conditions $(x_{0}, \vartheta_{0})$ and $k\in\mathbb{N},$
$
\mathbf{E}\{\|x(k,x_{0}, \vartheta_{0})\|^{2}|\vartheta_{0}=\ell\}
=x_{0}^{T}\mathcal{T}_{A}^{k}(\mathcal{I})(\ell)x_{0}
$
is valid for $\mu\text{-}$almost all $\ell\in{\Theta_{\vartheta_0}},$
where $x(k,x_{0}, \vartheta_{0})$ is the system state of $(A|\mathbb{G})$ with the initial conditions $(x_{0}, \vartheta_{0})$.
Furthermore, it is established that for any two different random variables $\hat{\vartheta}_{0}$ and $\bar{\vartheta}_{0}$,
\begin{equation}\label{377idenemssc}
\mathbf{E}\{\|x(k,x_{0},\hat{\vartheta}_{0})\|^{2}|\hat{\vartheta}_{0}=\ell\}
=\mathbf{E}\{\|x(k,x_{0},\bar{\vartheta}_{0})\|^{2}|\bar{\vartheta}_{0}=\ell\}
\end{equation}
is valid for $\mu\text{-}$almost all $\ell\in{\Theta_{\hat{\vartheta}_0}\bigcap\Theta_{\bar{\vartheta}_0}}.$

If $(A|\mathbb{G})$ is not EMSS-C, then
there exist initial conditions $(\bar{x}_{0}, \hat{\vartheta}_{0}),$
$\bar{k}\in\mathbb{N},$
and $\bar{\Lambda}\in\mathcal{B}(\Theta)$ with  $\mu(\bar{\Lambda})>0$
such that, for any $\beta\geq1$ and $\alpha\in(0,1)$, \begin{equation}\label{786fiu}
\mathbf{E}\{\|x(\bar{k},\bar{x}_{0}, \hat{\vartheta}_{0})\|^{2}|\hat{\vartheta}_{0}=\ell\} > \beta \alpha^{\bar{k}}\|\bar{x}_{0}\|^{2}
\end{equation}
holds for $\mu\text{-}$almost all $\ell\in\bar{\Lambda}$.
For set $\bar{\Lambda}$,
define a $\sigma$-algebraic $\bar{\Lambda}\bigcap\mathcal{B}(\Theta)
:=\{\bar{\Lambda}\bigcap\Delta,\ \forall \Delta\in\mathcal{B}(\Theta)\}.$
Clearly, $\mu$ is a measure on measurable space
$(\bar{\Lambda},\bar{\Lambda}\bigcap\mathcal{B}(\Theta)).$
Moreover, since $\mu$ is a $\sigma$-finite measure on $\mathcal{B}(\Theta)$,
there exist $\{\Lambda_{i},\ i\in\mathbb{N}^{+}\}$ with $\mu(\Lambda_{i})<+\infty$ such that
$\Theta =\bigcup _{i=1}^{\infty} \Lambda_{i}$.
Now we can write that
$$\bar{\Lambda}=\Theta \bigcap \bar{\Lambda}
=\bigcup _{i=1}^{\infty} (\Lambda_{i}\bigcap \bar{\Lambda}),$$ where $0\leq\mu(\Lambda_{i}\bigcap \bar{\Lambda})\leq \mu(\Lambda_{i})<+\infty$, and then
it can be concluded that
$\mu$ is a $\sigma$-finite measure on measurable space
$(\bar{\Lambda},\bar{\Lambda}\bigcap\mathcal{B}(\Theta)).$
Hence, there exist a measurable function $\hat{\nu}_{0}$ satisfying $\hat{\nu}_{0}(\ell)>0$ for almost all $\ell\in\bar{\Lambda}$ such that $\hat{\mu}_{0}(\bar{\Lambda})<+\infty$ (see the definition of $\sigma$-finite measure and Lemma 1.4 on Page 21 of \cite{Kallenberg2017}),
where $\hat{\mu}_{0}$
is defined by
$$\hat{\mu}_{0}(\Delta)
=\int_{\Delta}\hat{\nu}_{0}(\ell)\mu(d\ell)
\ \text{ for any } \Delta\in\bar{\Lambda}\bigcap\mathcal{B}(\Theta).$$
Define a measurable function $\bar{\nu}_{0}$ by
$$\bar{\nu}_{0}(\ell)=
\begin{cases}
[\hat{\nu}_{0}(\ell)]/[\hat{\mu}_{0}(\bar{\Lambda})], &  \ell\in \bar{\Lambda}, \\
0,&\ell\in \Theta\backslash \bar{\Lambda}.
\end{cases}
$$
Let $\bar{\vartheta}_{0}$ be a random variable with
the probability distribution $\bar{\mu}_{0}$, which is defined by
\begin{equation*}\label{637inditribution}
\bar{\mu}_{0}(\Lambda)=\mathbb{P}(\bar{\vartheta}_0\in{\Lambda})
=\int_{\Lambda}\bar{\nu}_{0}(\ell)\mu(d\ell)
\text{ for any } \Lambda\in\mathcal{B}(\Theta).
\end{equation*}
Clearly, the definition of $\bar{\mu}_{0}$ shows that $\bar{\mu}_{0}$ satisfies $(iv)$ of Assumption \ref{Assumption1}.
Combining \eqref{377idenemssc} and \eqref{786fiu}, we can conclude that
there exist initial conditions $(\bar{x}_{0}, \bar{\vartheta}_{0}),$
$\bar{k}\in\mathbb{N},$
and $\bar{\Lambda}\in\mathcal{B}(\Theta)$ with  $\mu(\bar{\Lambda})>0$
such that, for any $\beta\geq1$ and $\alpha\in(0,1)$,
\begin{equation*}
\begin{split}
\mathbf{E}\{\|x(\bar{k},\bar{x}_{0}, \bar{\vartheta}_{0})\|^{2}|\bar{\vartheta}_{0}=\ell\}
\bar{\nu}_{0}(\ell)
-{\beta \alpha^{\bar{k}}\|\bar{x}_{0}\|^{2}}{\bar{\nu}_{0}(\ell)}
>0
\end{split}
\end{equation*}
holds for $\mu\text{-}$almost all $\ell\in\bar{\Lambda}$.
Furthermore, according to Theorem 15.2 (ii) in \cite{BookBillingsley1995}, we can get that
there exist initial conditions $(\bar{x}_{0}, \bar{\vartheta}_{0})$ and
$\bar{k}\in\mathbb{N}$
such that, for any $\beta\geq1$ and $\alpha\in(0,1)$,
$
\mathbf{E}\{\|x(\bar{k},\bar{x}_{0}, \bar{\vartheta}_{0})\|^{2}\}
=\int_{\Theta}\mathbf{E}\{\|x(\bar{k},\bar{x}_{0}, \bar{\vartheta}_{0})\|^{2}|\bar{\vartheta}_{0}=\ell\}
\bar{\nu}_{0}(\ell)\mu(d\ell)
=\int_{\bar{\Lambda}}\mathbf{E}\{\|x(\bar{k},\bar{x}_{0}, \bar{\vartheta}_{0})\|^{2}|\bar{\vartheta}_{0}=\ell\}
\bar{\nu}_{0}(\ell)\mu(d\ell)
> {\beta \alpha^{\bar{k}}\|\bar{x}_{0}\|^{2}}
\int_{\bar{\Lambda}}{\bar{\nu}_{0}(\ell)}
\mu(d\ell)
=\beta \alpha^{\bar{k}}\|\bar{x}_{0}\|^{2}.$
This contradicts with  $(A|\mathbb{G})$ being EMSS.
The proof is thus completed.
\end{proof}
\begin{remark}
In Theorem 3.6 of \cite{BookDragan2010},
for time-varying MJLSs with a time-inhomogeneous finite-state Markov chain, the authors proposed a sufficient condition
to ensure the equivalence between
ESMS-CI and
exponential stability in the mean square.
This condition requires the existence of a Markov chain whose probability distribution satisfies a uniform positivity condition.
It is worth noting that the technique employed in proof of Theorem 3.6 from \cite{BookDragan2010} is inapplicable to the Borel space setting due to its uncountable nature. Moreover, it is shown in Theorem \ref{emss-emss-c} that for the case considered in this paper, the equivalence between EMSS and EMSS-C holds without any additional assumptions.
\end{remark}

According to Theorem \ref{emss-emss-c}, as well as Theorem 3 and Theorem 4 in \cite{Xiao2023},
the spectral criteria and the Lyapunov-type criteria for EMSS/EMSS-C are summarized as follows:
\begin{theorem}\label{EMSSCiff}
The following assertions are equivalent:

$(i)$~$(A|\mathbb{G})$ is EMSS;

$(ii)$~$r_{\sigma}(\mathcal{L}_{A})<1$;

$(iii)$~$(A|\mathbb{G})$ is EMSS-C;

$(iv)$~$r_{\sigma}(\mathcal{T}_{A})<1$;

$(v)$~Given any $V\in\mathcal{H}_{\infty}^{n+*}$, there exists $U\in\mathcal{H}_{\infty}^{n+*}$ such that
\begin{equation*}
U(\ell)-\mathcal{T}_{A}(U)(\ell)=V(\ell)\  \mu\text{-}a.e.;
\end{equation*}

$(vi)$~There exist $\xi>0$ and $U\in\mathcal{H}_{\infty}^{n+*}$  such that $U(\ell)-\mathcal{T}_{A}(U)(\ell)\geq \xi I$ $\mu\text{-}a.e.$.
\end{theorem}

To be sure, the equivalences in Theorem \ref{EMSSCiff} are consistent with the case where the Markov chain takes values in a countable set.

In the remainder of this paper, without ambiguity, we will only use the EMSS notion.

\subsection{The Decomposition of Matrix-Valued Functions in $\mathcal{H}_{\infty}^{n+}$ }\label{subsec3}
In this subsection, a useful lemma about the decomposition of matrix-valued functions in $\mathcal{H}_{\infty}^{n+}$ is given.
\begin{lemma}\label{semimeas864}
For any $P\in\mathcal{H}_{\infty}^{n+}$, there exists a unique $S\in\mathcal{H}_{\infty}^{n+}$
such that $P(\ell)=[S(\ell)]^{2}\ \mu\text{-}a.e.$.
Particularly, if $P\in\mathcal{H}_{\infty}^{n+*}$,
then there exists a unique $S\in\mathcal{H}_{\infty}^{n+*}$
such that $P(\ell)=[S(\ell)]^{2}\ \mu\text{-}a.e.$.
\end{lemma}
\begin{proof}
It needs only to prove that for any $P\in\mathcal{H}_{\infty}^{n+}$ with $\|P\|_{\infty}\neq 0$, the result is valid.
Otherwise, by setting $S=P$, it is immediate to get that $P(\ell)=[S(\ell)]^{2}\ \mu\text{-}a.e.$.
For any $P\in\mathcal{H}_{\infty}^{n+}$ with $\|P\|_{\infty}\neq 0$, one has that
$$0\leq  I-\frac{P(\ell)}{\|P\|_{\infty}}\leq I\  \mu\text{-}a.e..$$
Now for any $\ell\in\Theta$, define
$$Q_{k}(\ell):=\left\{
\begin{array}{ll}
\frac{1}{2}[I-\frac{P(\ell)}{\|P\|_{\infty}}],\ \ \ \ \ \ \ \ \ \ \ \ \ \ \ \  k=0, \\
\frac{1}{2}[I-\frac{P(\ell)}{\|P\|_{\infty}}+Q_{k-1}(\ell)^{2}],\ k\in\mathbb{N}^{+}.\\
\end{array}
\right.$$
By mathematical induction on $k$, it can be proved that $\|Q_{k}\|_{\infty}\leq 1$, $P Q_{k}=Q_{k} P$,
and $0\leq Q_{k}\leq  Q_{k+1}\leq \mathcal{I}$ hold for any $k\in\mathbb{N}$.
According to the monotone convergence theorem,
there exists $Q\in\mathcal{H}_{\infty}^{n+}$ such that $\lim_{k\rightarrow +\infty}Q_{k}=Q$ and
\begin{equation}\label{positive815}
Q(\ell)=\frac{1}{2}[I-\frac{P(\ell)}{\|P\|_{\infty}}+Q(\ell)^{2}].
\end{equation}
Setting $S=\|P\|_{\infty}^{\frac{1}{2}}(\mathcal{I}-Q)\in\mathcal{H}_{\infty}^{n+}$,
we conclude that $P(\ell)=[S(\ell)]^{2}\ \mu\text{-}a.e.$.
Additionally, for any $T\in\mathcal{SH}_{\infty}^{n}$ satisfying $P(\ell)T(\ell)=T(\ell)P(\ell)$,
it can be seen that for any $k\in\mathbb{N}$,
$T(\ell)Q_{k}(\ell)=Q_{k}(\ell)T(\ell)$,
and therefore $T(\ell)Q(\ell)=Q(\ell)T(\ell)$.
It leads to $T(\ell)S(\ell)=S(\ell)T(\ell)$.
Next we show the uniqueness of $S$.
Suppose there exists another $\bar{S}\in\mathcal{H}_{\infty}^{n+}$ such that
$P(\ell)=[\bar{S}(\ell)]^{2}\ \mu\text{-}a.e.$.
We derive that $\bar{S}(\ell)S(\ell)=S(\ell)\bar{S}(\ell)$.
Then, for any $x\in\mathbb{R}^{n}$,
$$x^{T}[S(\ell)-\bar{S}(\ell)][S(\ell)+\bar{S}(\ell)][S(\ell)-\bar{S}(\ell)]x=0.$$
Letting $S^{\frac{1}{2}}\in\mathcal{H}_{\infty}^{n+}$ and $\bar{S}^{\frac{1}{2}}\in\mathcal{H}_{\infty}^{n+}$
satisfy that  $S(\ell)=[S^{\frac{1}{2}}(\ell)]^{2}\ \mu\text{-}a.e.$ and $\bar{S}(\ell)=[\bar{S}^{\frac{1}{2}}(\ell)]^{2}\ \mu\text{-}a.e.$,
one can obtain that $S^{\frac{1}{2}}(\ell)[S(\ell)-\bar{S}(\ell)]x=0$
and $\bar{S}^{\frac{1}{2}}(\ell)[S(\ell)-\bar{S}(\ell)]x=0$.
Hence, $x^{T}[S(\ell)-\bar{S}(\ell)][S(\ell)-\bar{S}(\ell)]x=0$,
which follows that $S(\ell)=\bar{S}(\ell)\ \mu\text{-}a.e.$.\\
In particular, if $P\in\mathcal{H}_{\infty}^{n+*}$,
then there exists $\xi>0$ such that $P(\ell)\geq \xi I$ $\mu\text{-}a.e.$.
From \eqref{positive815} and the fact that $0\leq Q\leq \mathcal{I}$,
we draw that $I-Q(\ell)\geq \frac{\xi}{2\|P\|_{\infty}}I$ $\mu\text{-}a.e.$.
Thus, $S=\|P\|_{\infty}^{\frac{1}{2}}(\mathcal{I}-Q)\in\mathcal{H}_{\infty}^{n+*}$, and the proof is ended.
\end{proof}
\begin{remark}
We emphasize that in Lemma \ref{semimeas864}, $P,S\in\mathcal{H}_{\infty}^{n+}$,
that is, $P(\ell), S(\ell)\geq 0$ $\mu$-$a.e.$, and both $P(\ell)$ and  $S(\ell)$  are measurable matrix-valued functions that satisfy $\|P\|_{\infty}<\infty$ and $\|S\|_{\infty}<\infty$.
When the Markov chain takes values in a finite set, we have the eigenvalue decomposition $P=MDM^T$, where $M(\cdot)$ satisfies
$M(\cdot)^{T}M(\cdot)=M(\cdot)M(\cdot)^{T}=I$
and
$D(\cdot)$ is the diagonal matrix with the non-negative eigenvalues of $P(\cdot);$
%$D(\cdot)=diag\{\begin{array}{cccc}
%                       \lambda_{1}(\cdot), & \lambda_{2}(\cdot), & \cdots, & \lambda_{n}(\cdot)
%                     \end{array}\}$
%with $\lambda_{i}(\cdot)$, $i\in\overline{1,n}$ being the non-negative eigenvalues of $P(\cdot)$;
further, the square root decomposition via $P=S^{2}$ is standard, where $S=MD^{1/2}M^T$ and $D^{1/2}(\cdot)$ is the diagonal matrix with the square roots of the non-negative eigenvalues of $P(\cdot)$.
However, this decomposition could not be naturally generalized to
the case where the Markov chain takes values in a Borel set,
for the reason that the measurability of the matrix-valued functions $M(\ell)$ and $D^{1/2}(\ell)$
cannot be trivially deduced from the measurability of $P(\ell)$, $\ell\in\Theta$.
To bridge over the difficulties,
in the proof of Lemma \ref{semimeas864}, we have constructed a  nondecreasing sequence of matrix-valued functions $\{Q_k,k\in\mathbb{N}\}$
with $Q_k\in\mathcal{H}_{\infty}^{n+}$.
And then based on the fact that the limit of
a sequence of measurable functions remains measurable, the desired $S\in\mathcal{H}_{\infty}^{n+}$ has been created.
\end{remark}

\section{Detectability}\label{Detectability}
This section devotes to dealing with detectability of the unforced system.
We start with the definition of stabilizability.

Consider the MJLSs
$$x(k+1)=A(\vartheta(k))x(k)+G(\vartheta(k))u(k),\ \forall k\in\mathbb{N},$$
or $(A, G|\mathbb{G})$ for short, where $G=\{G(\ell)\}_{\ell\in\Theta}\in\mathcal{H}^{n\times m}_{\infty}$.
\begin{definition}\label{589defstabilizability}
$(A, G|\mathbb{G})$ is said to be exponentially mean-square stabilizable
if there exists $K\in\mathcal{H}_{\infty}^{m\times n}$ such that $(A+GK|\mathbb{G})$ is EMSS.
In this case, we call that $u(k)=K(\vartheta(k))x(k)$ exponentially stabilizes $(A, G|\mathbb{G})$
and the set of all stabilizing feedback gains is denoted by $\mathbf{K}$.
\end{definition}

Denote the following MJLSs by $(A;C|\mathbb{G})$:
$$\left\{
\begin{array}{ll}
x(k+1)=A(\vartheta(k))x(k),\ \\
z(k)=C(\vartheta(k))x(k),\ \forall k\in\mathbb{N}.
\end{array}
\right.$$

\begin{definition}\label{defexde1009}
$(A;C|\mathbb{G})$ is said to be detectable if there exists $H\in\mathcal{H}^{n\times p}_{\infty}$ such that $x(k+1)=[A(\vartheta(k))
+H(\vartheta(k))C(\vartheta(k))]x(k)$
 is EMSS.
\end{definition}

The following theorem presents a criterion for the exponential stability of $(A|\mathbb{G})$
under the assumption that $(A;C|\mathbb{G})$ is detectable.

\begin{theorem}\label{Mthe1495if}
Suppose that $(A;C|\mathbb{G})$ is detectable.
$(A|\mathbb{G})$ is EMSS iff there exists $U\in\mathcal{H}^{n+}_{\infty}$ such that
\begin{equation}\label{1157utaucc}
U(\ell)-\mathcal{T}_{A}(U)(\ell)=C(\ell)^{T}C(\ell)\ \mu\text{-}a.e..
\end{equation}
\end{theorem}
\begin{proof}
Since Theorem 5 in \cite{Xiao2023} furnished the proof for the ``only if'' part,
it remains to confirm the ``if'' part.
Because of the detectability of $(A;C|\mathbb{G})$,
there exists $H\in\mathcal{H}^{n\times p}_{\infty}$ satisfying $r_{\sigma}\{\mathcal{L}_{A+HC}\}<1$.
Let $\widehat{\mathcal{L}}=(1+\xi^{2})\mathcal{L}_{A+HC}$ and choose a sufficiently small $\xi\neq0$ such that $r_{\sigma}\{\widehat{\mathcal{L}}\}<1$.
By Proposition \ref{lemma468L}, for any $\bar{V}\in \mathcal{H}_{1}^{n+}$, it is true that
\begin{equation}\label{643ahc}
\mathcal{L}_{A}(\bar{V})(\ell)\leq\widehat{\mathcal{L}}(\bar{V})(\ell)
+\widetilde{\mathcal{L}}(\bar{V})(\ell)\ \mu\text{-}a.e.,
\end{equation}
where $\widetilde{\mathcal{L}}=(1+{1}/{\xi^{2}})\mathcal{L}_{HC}$.
Combining this with \eqref{88Xklre},
for a given sequence $\{X(k),k\in\mathbb{N}\}$, where $X(k)\in\mathcal{H}_{1}^{n+}$ is defined by \eqref{BArxX},
it holds that
$$
X(k+1)(\ell)\leq\widehat{\mathcal{L}}(X(k))(\ell)
+\widetilde{\mathcal{L}}(X(k))(\ell)\ \mu\text{-}a.e..
$$
Define $\{Q(k),k\in\mathbb{N}\}$ as follows:
$$
\left\{
\begin{array}{ll}
Q(k+1)(\ell)=\widehat{\mathcal{L}}(Q(k))(\ell)
+\widetilde{\mathcal{L}}(X(k))(\ell),\\
Q(0)(\ell)=X_{0}(\ell),\ \ell\in\Theta.\\
\end{array}
\right.
$$
Inductively, one can deduce that for each $k\in\mathbb{N},$ $Q(k)\in\mathcal{H}^{n+}_{1}$, and
\begin{equation}\label{1216qx-solu}
Q(k+1)(\ell)=\widehat{\mathcal{L}}^{k+1}(X_{0})(\ell)
+\sum_{i=0}^{k}\widehat{\mathcal{L}}^{k-i}[\widetilde{\mathcal{L}}(X(i))](\ell)\ \mu\text{-}a.e..
\end{equation}
We will prove that $X(k)\leq Q(k)$ is valid for each $k\in\mathbb{N}$.
To this end, define $Z(k)=Q(k)-X(k),\ k\in\mathbb{N}$.
$Z(0)\in\mathcal{H}^{n+}_{1}$ is trivially satisfied, and it just needs to show that
$Z(k+1)\in\mathcal{H}^{n+}_{1}$ holds for every $k\in\mathbb{N}$.
It can be calculated that for any $k\in\mathbb{N}$,
$Z(k+1)(\ell)=\widehat{\mathcal{L}}(Z(k))(\ell)+\underline{\mathcal{L}}(X(k))(\ell)\ \mu\text{-}a.e.,$
where $\widehat{\mathcal{L}}$ and $\underline{\mathcal{L}}=\widehat{\mathcal{L}}+\widetilde{\mathcal{L}}-\mathcal{L}_{A}$
are positive operators on $\mathcal{SH}^{n}_{1}$ by $(i)$ of Proposition \ref{adjoint} and \eqref{643ahc}.
Via induction, for each $k\in\mathbb{N}$,
\begin{equation*}\label{1071zksemipo}
Z(k+1)(\ell)=\widehat{\mathcal{L}}^{k+1}(Z(0))(\ell)
+\sum_{i=0}^{k}\widehat{\mathcal{L}}^{k-i}[\underline{\mathcal{L}}(X(i))](\ell)\ \mu\text{-}a.e..
\end{equation*}
Further, $Z(k+1)\in\mathcal{H}^{n+}_{1}$,
which means that $X(k+1)\leq Q(k+1)$, $\forall k\in\mathbb{N}$.
Invoking \eqref{1216qx-solu}, one can obtain that for each $k\in\mathbb{N}$,
\begin{equation}\label{1213xQjkkkk}
\begin{aligned}
\|X(k+1)\|_{1}
&\leq \|Q(k+1)\|_{1}\\
&\leq  \|\widehat{\mathcal{L}}^{k+1}(X_{0})\|_{1}
+\sum_{i=0}^{k}\|\widehat{\mathcal{L}}^{k-i}\|
\|\widetilde{\mathcal{L}}(X(i))\|_{1}.
\end{aligned}
\end{equation}
In view of $r_{\sigma}(\widehat{\mathcal{L}})<1$, we deduce that
there exist $\beta \geq 1$, $\alpha \in(0,1)$ such that
$\|\widehat{\mathcal{L}}^{k}\| \leq \beta \alpha^{k}$ holds for each $k\in{\mathbb{N}}$.
$\|\widetilde{\mathcal{L}}(X(i)\|_{1}\leq \Gamma \|CX(i)C^{T}\|_{1}$ is then proved by using Fubini's theorem,
where $\Gamma=(1+{1}/{\xi^{2}})\|H\|_{\infty}^{2}$.
Hence, according to Lemma \ref{lemma345XxXxl2}, we get that
$\|\widetilde{\mathcal{L}}(X(i))\|_{1}\leq\Gamma\mathbf{E}\{\|C(\vartheta(i))x(i)\|^{2}\}.$
Summing \eqref{1213xQjkkkk} over $k$ yields that
\begin{equation}\label{1253sumxk}
\begin{aligned}
\sum_{k=0}^{\infty}\|X(k)\|_{1}\leq \frac{\beta}{1-\alpha}\{\|X_{0}\|_{1}
+\Gamma\sum_{i=0}^{\infty}\mathbf{E}[\|C(\vartheta(i))x(i)\|^{2}]\}.
\end{aligned}
\end{equation}
Under the hypothesis that there exists $U\in\mathcal{H}^{n+}_{\infty}$ satisfying \eqref{1157utaucc},
one can derive that
\begin{equation*}
\begin{split}
\mathbf{E}\{x_{0}^{T}U(\vartheta_{0})x_{0}\}
=&\sum_{i=0}^{k}\mathbf{E}\{x(i)^{T}C(\vartheta(i))^{T}C(\vartheta(i))x(i)\}\\
&+\mathbf{E}\{x(k+1)^{T}U(\vartheta(k+1))x(k+1)\}.
\end{split}
\end{equation*}
Taking the limits for $k\rightarrow\infty$, it leads to
\begin{equation*}
\begin{split}
\sum_{i=0}^{\infty}\mathbf{E}\{\|C(\vartheta(i))x(i)\|^{2}\}&\leq\mathbf{E}\{x_{0}^{T}U(\vartheta_{0})x_{0}\}\\
&\leq\eta\|x_{0}\|^{2}\leq n\eta\|X_{0}\|_{1},
\end{split}
\end{equation*}
where $\eta=\|U\|_{\infty}$.
From this, \eqref{1253sumxk}, and \eqref{495xk=Lk},
it arrives at a conclusion that for any  $X_{0}\in\mathcal{H}_{1}^{n+}$ defined by $X_{0}(\ell)=x_{0}x_{0}^{T}\nu_{0}(\ell)$,
\begin{equation}\label{1284xxx0}
\sum_{k=0}^{\infty}\|\mathcal{L}_{A}^{k}(X_{0})\|_{1}\leq\Gamma_{0}\|X_{0}\|_{1},
\end{equation}
where $\Gamma_{0}=\frac{\beta}{1-\alpha}(1+n\eta\Gamma)>0$.
Next, we will show that for any $\bar{V} \in \mathcal{H}_{1}^{n+}$,
$\sum_{k=0}^{\infty}\|\mathcal{L}^{k}_{A}(\bar{V})\|_{1}<\infty$.
If $\|\bar{V}\|_{1}=0$, it is clear that $\sum_{k=0}^{\infty}\|\mathcal{L}^{k}_{A}(\bar{V})\|_{1}<\infty$.
Consider the case that $\|\bar{V}\|_{1}\neq 0$.
Define $\hat{V}\in\mathcal{H}_{1}^{n+}$ with $\hat{V}(\ell)=\|\bar{V}(\ell)\|I$.
It follows that $\bar{V}\leq \hat{V}$ and for every $k\in\mathbb{N}$,
$\|\mathcal{L}_{A}^{k}(\bar{V})\|_{1} \leq\|\mathcal{L}_{A}^{k}(\hat{V})\|_{1}$.
For each $i\in \overline{1,n}$,
define the $i$-th unit vector $e_{i}\in \mathbb{R}^{n}$ with all components being 0 except for the $i$-th component, which is 1.
Note that $\hat{V}$ can be written as $\hat{V}=\sum_{i=1}^{n} \hat{V}_{i}$ with $\hat{V}_{i}=\{\hat{V}_{i}(\ell)\}_{\ell\in\Theta}$,
where $\hat{V}_{i}(\ell)=\|\bar{V}(\ell)\|e_{i}e_{i}^{T}$.
By the linearity of $\mathcal{L}_{A}$, we have that for each $k\in\mathbb{N}$,
\begin{equation}\label{1529KAHATPII}
\|\mathcal{L}_{A}^{k}(\hat{V})\|_{1}\leq \sum_{i=1}^{n}\|\mathcal{L}_{A}^{k}(\hat{V}_{i})\|_{1}.
\end{equation}
Clearly, for each $i\in \overline{1,n}$, $\|\hat{V}_{i}\|_{1}=\|\bar{V}\|_{1}$.
In addition, there exist $\nu_{0}$ with ${\nu}_{0}(\ell)=\|\bar{V}(\ell)\|/\|\bar{V}\|_{1}$ satisfying Assumption \ref{Assumption1},
and an initial state $x_{0}^{i}=\|\bar{V}\|_{1}^{\frac{1}{2}}e_{i}\in\mathbb{R}^{n}$
such that $\hat{V}_{i}=X_{0}^{i}$ with $X_{0}^{i}(\ell)=x_{0}^{i} (x_{0}^{i})^{T}\nu_{0}(\ell).$
Then, we can rewrite \eqref{1529KAHATPII} as
\begin{equation}\label{1535LAkP1IN}
\|\mathcal{L}_{A}^{k}(\hat{V})\|_{1}\leq \sum_{i=1}^{n}\|\mathcal{L}_{A}^{k}(X_{0}^{i})\|_{1}
\end{equation}
and get that $\|X_{0}^{i}\|_{1}=\|\hat{V}_{i}\|_{1}=\|\bar{V}\|_{1}$.
Additionally, it can be deduced from \eqref{1284xxx0} that for each $i\in \overline{1,n}$,
there exist $\beta \geq 1$, $\alpha \in(0,1)$ such that
$\sum_{k=0}^{\infty}\|\mathcal{L}_{A}^{k}(X_{0}^{i})\|_{1}\leq\Gamma_{0}\|\bar{V}\|_{1}.$
This, together with \eqref{1535LAkP1IN}, reveals that
$\sum_{k=0}^{\infty}\|\mathcal{L}_{A}^{k}(\hat{V})\|_{1} \leq n\Gamma_{0}\|\bar{V}\|_{1}$.
Noticing the fact that $\|\mathcal{L}_{A}^{k}(\bar{V})\|_{1} \leq\|\mathcal{L}_{A}^{k}(\hat{V})\|_{1}$,
it can be drawn that for any $\bar{V} \in \mathcal{H}_{1}^{n+}$,
$\sum_{k=0}^{\infty}\|\mathcal{L}_{A}^{k}(\bar{V})\|_{1} \leq n\Gamma_{0}\|\bar{V}\|_{1}.$
By Proposition 2 in \cite{Xiao2023}, for any $V\in\mathcal{SH}^{n}_{1}$,
there exist $V^{+}$, $V^{-}\in \mathcal{H}_{1}^{n+}$ such that
$V=V^{+}-V^{-}$, $\max\{\|V^{+}\|_{1},\|V^{-}\|_{1}\}\leq \|V\|_{1}$, and
\begin{equation*}\label{1523pp=-p-}
\sum_{k=0}^{\infty}\|\mathcal{L}^{k}_{A}(V)\|_{1} \leq \sum_{k=0}^{\infty}\|\mathcal{L}^{k}_{A}(V^{+})\|_{1}
+\sum_{k=0}^{\infty}\|\mathcal{L}^{k}_{A}(V^{-})\|_{1}.
\end{equation*}
Therefore, for any $V\in\mathcal{SH}^{n}_{1}$,
$\sum_{k=0}^{\infty}\|\mathcal{L}_{A}^{k}(V)\|_{1}\leq2n\Gamma_{0}\|V\|_{1}$.
Applying $(b)\Leftrightarrow(c)$ of Lemma 2.1 in \cite{Przyluski1988} and Theorem \ref{EMSSCiff},
we assure that $(A|\mathbb{G})$ is EMSS, which ends the proof.
\end{proof}

\section{The Maximal/Stabilizing Solution of coupled-AREs}\label{maxstaGCARE}
In this section, we will concentrate on the maximal solution and the stabilizing solution of
coupled-AREs associated with the LQ optimization problem for MJLSs.

In what follows, for any $P\in\mathcal{SH}_{\infty}^{n}$, $Q\in\mathcal{SH}^{n}_{\infty}$, $R\in{\mathcal{H}_{\infty}^{m+*}}$, and $\ell\in\Theta$,
define
\begin{align}
&\mathcal{G}(P)(\ell):=G(\ell)^{T} \mathcal{E}(P)(\ell)A(\ell),\nonumber \\
&\mathcal{R}(P)(\ell):=R(\ell)+ \mathcal{T}_{G}(P)(\ell),\nonumber \\ &\mathcal{M}(P)(\ell):=-\mathcal{R}(P)(\ell)^{-1}
\mathcal{G}(P)(\ell),\nonumber \\
&\mathcal{W}(P)(\ell):=-P(\ell)+\mathcal{T}_{A}(P)(\ell)+Q(\ell)
\nonumber \\
&\ \ \ \ \ \ \ \ \ \ \ \ \ \ \ -\mathcal{G}(P)(\ell)^{T}\mathcal{R}(P)(\ell)^{-1}\mathcal{G}(P)(\ell),\nonumber
\end{align}
and
\begin{align}
&\mathbf{S}_{+}:=\{P\in\mathcal{SH}_{\infty}^{n}| \mathcal{R}(P)(\ell)\gg 0 \  \mu\text{-}a.e.\nonumber \\
&\ \ \ \ \ \ \ \ \ \ \ \ \ \ \ \ \ \ \ \ \ \ \ \ \ \ \ \ \ \ \ \ \text{and}\ \mathcal{W}(P)(\ell)\geq 0\  \mu\text{-}a.e.\}.\nonumber
\end{align}
It can be easily checked that $\mathcal{G}\in\mathbf{B}(\mathcal{SH}_{\infty}^{n},\mathcal{H}_{\infty}^{m\times n})$, $\mathcal{R}\in\mathbf{B}(\mathcal{SH}_{\infty}^{n}, \mathcal{SH}_{\infty}^{m})$, $\mathcal{M}\in\mathbf{B}(\mathcal{SH}_{\infty}^{n},\mathcal{H}_{\infty}^{m\times n})$,
and $\mathcal{W}\in\mathbf{B}(\mathcal{SH}_{\infty}^{n}, \mathcal{SH}_{\infty}^{n})$.

Consider the following coupled-AREs:
\begin{equation}\label{GCARE1}
P(\ell)=\mathcal{T}_{A}(P)(\ell)+Q(\ell)
-\mathcal{G}(P)(\ell)^{T}\mathcal{R}(P)(\ell)^{-1}\mathcal{G}(P)(\ell)
\end{equation}
with the sign condition
\begin{equation}\label{SignCon771}
\mathcal{R}(P)(\ell)\gg 0\ \mu\text{-}a.e..
\end{equation}

The following definitions generalize the well-known concepts of the maximal solution and the stabilizing solution of Riccati equations.
\begin{definition}\label{solutionmaxst993}
$(i)$~A solution $P_{max}\in \mathcal{SH}_{\infty}^{n}$ is said to be the maximal solution of coupled-AREs \eqref{GCARE1} over $\mathbf{S}_{+}$
if it satisfies $P\leq P_{max}$ for any $P\in{\mathbf{S}_{+}};$

$(ii)$~A solution $P_{st}\in\mathcal{SH}_{\infty}^{n}$ is said to be the stabilizing solution of coupled-AREs \eqref{GCARE1}
if $u(k)=\mathcal{M}(P_{st})(\vartheta(k)) x(k)$, $k\in\mathbb{N}$ exponentially stabilizes $(A, G|\mathbb{G})$.
\end{definition}

The next lemma is an immediate consequence of Theorem 2.5 in \cite{BookDragan2010}.
\begin{lemma}\label{1116exikx}
If $(A, G|\mathbb{G})$ is exponentially mean-square stabilizable,
then there exists $K\in\mathcal{H}_{\infty}^{m\times n}$ satisfying $K\in\mathbf{K}$.
Moreover, for any $Q\in\mathcal{SH}^{n}_{\infty}$ and $R\in{\mathcal{H}_{\infty}^{m+*}}$,
there exists a unique $P\in{\mathcal{SH}_{\infty}^{n}}$ such that
\begin{equation}\label{815ni}
P(\ell)\!-\!\mathcal{T}_{A+GK}(P)(\ell)\!=\!Q(\ell)\!+\!K(\ell)^{T}R(\ell)K(\ell)\ \mu\text{-}a.e..
\end{equation}
\end{lemma}
For clarity, given $K\in\mathbf{K}$, we denote the unique $P\in{\mathcal{SH}^{n}_{\infty}}$ satisfying \eqref{815ni} by $P_K$.

The following theorem presents a sufficient condition for the existence of the maximal solution of coupled-AREs \eqref{GCARE1} over $\mathbf{S}_+$.
\begin{theorem}\label{1145maxso}
Suppose that $(A, G|\mathbb{G})$ is exponentially mean-square stabilizable.
If $\mathbf{S}_+\neq\emptyset$,
then there exists a maximal solution of coupled-AREs \eqref{GCARE1} satisfying the sign condition \eqref{SignCon771}.
\end{theorem}
\begin{proof}
According to Lemma \ref{1116exikx}, there exists $K\in\mathcal{H}_{\infty}^{m\times n}$ satisfying $K\in\mathbf{K}$.
Now we inductively prove that for any $h\in\mathbb{N}$,
there exist $P_{h}\in{\mathcal{SH}_{\infty}^{n}}$ and $K_{h}\in\mathcal{H}_{\infty}^{m\times n}$ such that the following hold:

$(i)$~For any $P\in\mathbf{S}_{+}$, $P\leq P_{h}\leq P_{h-1}\leq \cdots\leq  P_{1}\leq P_{0}$;

$(ii)$~$u(k)=K_{h}(\vartheta(k))x(k),\ k\in\mathbb{N}$ exponentially stabilizes $(A, G|\mathbb{G})$, where
\begin{equation}\label{KK946}
K_{h}:=\left\{
\begin{array}{ll}
K,\ \ \ \ \ \ \ \ \ \  h=0, \\
\mathcal{M}(P_{h-1}),\ h\in\mathbb{N}^{+};\\
\end{array}
\right.
\end{equation}

$(iii)$~$P_{h}=P_{K_{h}}$.\\
Set $K_{0}=K$.
By Lemma \ref{1116exikx}, there exists a unique $P_{K_{0}}\in{\mathcal{SH}_{\infty}^{n}}$ satisfying \eqref{815ni}.
Let $P_{0}=P_{K_{0}}$ and $\bar{K}=\mathcal{M}(P)$.
Since for any $P\in\mathbf{S}_{+}$,
\begin{equation*}
\begin{split}
&P(\ell)-\mathcal{T}_{A+GK_{0}}(P)(\ell)=Q(\ell)+K_{0}(\ell)^{T}R(\ell)K_{0}(\ell)\\
&-\mathcal{W}(P)(\ell)-[K_{0}(\ell)-\bar{K}(\ell)]^{T}\mathcal{R}(P)(\ell)[K_{0}(\ell)-\bar{K}(\ell)],
\end{split}
\end{equation*}
it can be calculated that
\begin{equation*}
\begin{split}
&[P_{0}(\ell)-P(\ell)]-\mathcal{T}_{A+GK_{0}}[P_{0}(\ell)-P(\ell)]\\
=&\mathcal{W}(P)(\ell)+[K_{0}(\ell)-\bar{K}(\ell)]^{T}\mathcal{R}(P)(\ell)[K_{0}(\ell)-\bar{K}(\ell)],
\end{split}
\end{equation*}
which yields by Theorem \ref{Mthe1495if} that $P\leq P_{0}$.
Thus, $(i)$-$(iii)$ hold for $h=0$.

Assume that $(i)$-$(iii)$ are valid for $h-1$.
Let $K_{h}=\mathcal{M}(P_{h-1})$.
For any $P\in{\mathbf{S}_{+}}$, from $P\leq P_{h-1}$, we have that $\mathcal{R}(P_{h-1})\in\mathcal{H}_{\infty}^{m+*}$.
Applying Lemma \ref{semimeas864}, it can be ensured that there exist
$\mathcal{W}(P)^{\frac{1}{2}}\in\mathcal{H}_{\infty}^{n+}$,
$\mathcal{R}(P)^{\frac{1}{2}}\in\mathcal{H}_{\infty}^{m+*}$,
and $\mathcal{R}(P_{h-1})^{\frac{1}{2}}\in\mathcal{H}_{\infty}^{m+*}$ such that
$\mathcal{W}(P)(\ell)=[\mathcal{W}(P)^{\frac{1}{2}}(\ell)]^{2}$,
$\mathcal{R}(P)(\ell)=[\mathcal{R}(P)^{\frac{1}{2}}(\ell)]^{2}$, and
$\mathcal{R}(P_{h-1})(\ell)=[\mathcal{R}(P_{h-1})^{\frac{1}{2}}(\ell)]^{2}\ \mu\text{-}a.e..$
Because $u(k)=K_{h-1}(\vartheta(k))x(k)$ exponentially stabilizes $(A, G|\mathbb{G})$,
it follows that $(A+GK_{h};\mathcal{R}(P_{h-1})^{\frac{1}{2}}(K_{h}-K_{h-1})|\mathbb{G})$ is detectable.
Then, $(A+GK_{h};C_{1}|\mathbb{G})$ is detectable, where
$$C_{1}(\ell)=\left[
               \begin{array}{c}
                \mathcal{W}(P)^{\frac{1}{2}}(\ell) \\
                 {\mathcal{R}(P)}^{\frac{1}{2}}(\ell)
                 [K_{h}(\ell)-\bar{K}(\ell)]\\
                 {\mathcal{R}(P_{h-1})}^{\frac{1}{2}}(\ell)
                 [K_{h}(\ell)-K_{h-1}(\ell)] \\
               \end{array}
             \right].$$
On the other hand, from
\begin{equation*}
\begin{split}
&P_{h-1}(\ell)-\mathcal{T}_{A+GK_{h}}(P_{h-1})(\ell)=Q(\ell)+K_{h}(\ell)^{T}R(\ell) K_{h}(\ell)\\
&+[K_{h}(\ell)-K_{h-1}(\ell)]^{T}\mathcal{R}(P_{h-1})(\ell)[K_{h}(\ell)-K_{h-1}(\ell)]
\end{split}
\end{equation*}
and
\begin{equation}\label{1204ppp}
\begin{split}
&P(\ell)\!-\!\mathcal{T}_{A+GK_{h}}(P)(\ell)\!=\!Q(\ell)\!+\!K_{h}(\ell)^{T}R(\ell)K_{h}(\ell)\\
&\ \ \ \ \ \!-\!\mathcal{W}(P)(\ell)-[K_{h}(\ell)\!-\!\bar{K}(\ell)]^{T}\mathcal{R}(P)(\ell)[K_{h}(\ell)\!-\!\bar{K}(\ell)],
\end{split}
\end{equation}
it is made that
\begin{equation}\label{1208pqq}
[P_{h-1}(\ell)-P(\ell)]-\mathcal{T}_{A+GK_{h}}[P_{h-1}(\ell)-P(\ell)]=C_{1}(\ell)^{T}C_{1}(\ell).
\end{equation}
In view of $P\leq P_{h-1}$, and using Theorem \ref{Mthe1495if},
we know that $u(k)=K_{h}(\vartheta(k)) x(k)$ exponentially stabilizes $(A, G|\mathbb{G})$.
Let $P_{h}=P_{K_{h}}$, or equivalently,
\begin{equation}\label{1229PKLIM}
P_{h}(\ell)-\mathcal{T}_{A+GK_{h}}(P_{h})(\ell)=Q(\ell)+K_{h}(\ell)^{T}R(\ell)K_{h}(\ell).
\end{equation}
Noting that \eqref{1229PKLIM} minus \eqref{1204ppp} leads to
\begin{equation}\label{1233pkekkk}
\begin{split}
[P_{h}(\ell)-P(\ell)]-\mathcal{T}_{A+GK_{h}}[P_{h}(\ell)-P(\ell)]=C_{2}(\ell)^{T}C_{2}(\ell),
\end{split}
\end{equation}
where
$$C_{2}(\ell)=\left[
               \begin{array}{c}
                \mathcal{W}(P)^{\frac{1}{2}}(\ell) \\
                 {\mathcal{R}(P)}^{\frac{1}{2}}(\ell)[K_{h}(\ell)-\bar{K}(\ell)]\\
               \end{array}
             \right],$$
we conclude by Theorem \ref{Mthe1495if} that $P\leq P_{h}$.
Similarly, subtracting \eqref{1233pkekkk} from \eqref{1208pqq} gives
\begin{equation*}
\begin{split}
&[P_{h-1}(\ell)-P_{h}(\ell)]-\mathcal{T}_{A+GK_{h}}[P_{h-1}(\ell)-P_{h}(\ell)]\\
=&[K_{h}(\ell)-K_{h-1}(\ell)]^{T}
\mathcal{R}(P_{h-1})(\ell)[K_{h}(\ell)-K_{h-1}(\ell)].
\end{split}
\end{equation*}
This, together with the fact that $u(k)=K_{h}\left(\vartheta(k)\right)x(k)$
exponentially stabilizes $(A, G|\mathbb{G})$, results in $P_{h}\leq P_{h-1}$.
Therefore, $(i)$-$(iii)$ hold for any $h\in\mathbb{N}$.

Based on the monotone convergence theorem,
it follows from $(i)$ that there exists $\lim_{h\rightarrow +\infty}P_{h}=P_{max}\in{\mathcal{SH}_{\infty}^{n}}$
such that $P\leq P_{max}$ and $\mathcal{R}(P_{max})(\ell)\gg 0\ \mu\text{-}a.e..$
So, $\mathcal{W}(P_{max})=0$ is reached by taking $h\rightarrow+\infty$ and rearranging \eqref{1229PKLIM}.
The proof is then ended.
\end{proof}

\begin{remark}
Actually, by the definition of $\mathbf{S}_+$ and Theorem \ref{1145maxso},
it can be observed that under the hypothesis that $(A, G|\mathbb{G})$ is exponentially mean-square stabilizable,
$\mathbf{S}_+\neq\emptyset$ is equivalent to the existence of the maximal solution of coupled-AREs \eqref{GCARE1}.
\end{remark}

We are now in a position to cope with the existence of the stabilizing solution of coupled-AREs \eqref{GCARE1},
which is of great significance in designing the optimal and stabilizing controller related to some LQ optimization problems.

\begin{theorem}\label{1160exmax}
The following assertions are equivalent:

$(i)$~$(A, G|\mathbb{G})$ is exponentially mean-square stabilizable and there exists $P\in{\mathbf{S}_+}$ such that $(A+G\mathcal{M}(P);\mathcal{W}(P)^{\frac{1}{2}}|\mathbb{G})$ is detectable;

$(ii)$~The stabilizing solution of coupled-AREs \eqref{GCARE1} exists and satisfies the sign condition \eqref{SignCon771}.
\end{theorem}
\begin{proof}
$(i)\rightarrow(ii)$:
From Theorem \ref{1145maxso}, it follows that there exists $P_{max}\in{\mathbf{S}_+}$
such that $\mathcal{W}(P_{max})(\ell)=0 \ \mu\text{-}a.e.$
and $P\leq P_{max}$ for any $P\in{\mathbf{S}_+}$.
Now choosing an arbitrarily fixed $P\in{\mathbf{S}_+}$ such that $(A+G\mathcal{M}(P);\mathcal{W}(P)^{\frac{1}{2}}|\mathbb{G})$ is detectable,
we know that $(A+G\mathcal{M}(P);C_{3}|\mathbb{G})$ is detectable,
where
$$C_3(\ell)=\left[
  \begin{array}{c}
  \mathcal{W}(P)^{\frac{1}{2}}(\ell) \\
   {\mathcal{R}(P)^{\frac{1}{2}}}(\ell)
   [\mathcal{M}(P_{max})(\ell)-\mathcal{M}(P)(\ell)] \\
    \end{array}
    \right].$$
Noting that
\begin{equation*}
\begin{split}
&[P_{max}(\ell)-P(\ell)]-\mathcal{T}_{A+G\mathcal{M}(P_{max})}
[P_{max}(\ell)-P(\ell)]\\
=&C_3(\ell)^{T}C_3(\ell),
\end{split}
\end{equation*}
and by Theorem \ref{Mthe1495if}, it is concluded that $P_{max}$ is the stabilizing solution of coupled-AREs \eqref{GCARE1}.

$(ii)\rightarrow(i)$: If coupled-AREs \eqref{GCARE1} admit a stabilizing solution $P_{st}$ satisfying \eqref{SignCon771},
then $P_{st}\in{\mathbf{S}_+}$ and $(A+ G\mathcal{M}(P_{st})|\mathbb{G})$ is EMSS.
Therefore $(A, G|\mathbb{G})$ is exponentially mean-square stabilizable
and $(A+G\mathcal{M}(P_{st});\mathcal{W}(P_{st})^{\frac{1}{2}}|\mathbb{G})$ is detectable.
\end{proof}

\begin{remark}
Theorem \ref{1160exmax} makes it known that under the conditions of exponential mean-square stabilizability and detectability,
the maximal solution of coupled-AREs \eqref{GCARE1} is just the stabilizing solution.
On the other hand, by a similar argument with \cite{BookDragan2010},
one can also prove that the stabilizing solution of coupled-AREs \eqref{GCARE1},
if exists, must be the maximal solution.
Moreover, if the stabilizing solution of coupled-AREs \eqref{GCARE1} exists, then it is always unique.
\end{remark}

The following corollary presents a condition that guarantees the existence of the stabilizing solution of coupled-AREs \eqref{GCARE1}
associated with the standard LQ optimization problem.
\begin{corollary}\label{1234stLQ}
Fix $Q\in\mathcal{H}_{\infty}^{n+}$ and $R\in{\mathcal{H}_{\infty}^{m+*}}$.
Suppose that $(A, G|\mathbb{G})$ is exponentially mean-square stabilizable and $(A;Q^{\frac{1}{2}}|\mathbb{G})$ is detectable,
then coupled-AREs \eqref{GCARE1} admits a unique stabilizing solution satisfying \eqref{SignCon771}.
\end{corollary}
\begin{proof}
For any $\ell\in\Theta$, set $P_{0}(\ell)=0$.
Apparently, $\mathcal{W}(P_{0})(\ell)=Q(\ell)\geq 0$ and $ \mathcal{R}(P_{0})(\ell)=R(\ell)\gg 0 \ \mu\text{-}a.e.$.
Therefore, $P_0\in\mathbf{S}_+$, and the desired result is derived from Theorem \ref{1160exmax}.
\end{proof}
\begin{remark}
From the proofs of Theorems \ref{1145maxso} and \ref{1160exmax},
it can be observed that, under certain conditions,
the solutions $\{P_{h},\ h\in\mathbb{N}\}$ of the generalized Lyapunov equations \eqref{KK946} and \eqref{1229PKLIM}
converge to the stabilizing solution of coupled-AREs \eqref{GCARE1}.
Enlightened by that, one can design an iterative procedure by recursively solving \eqref{KK946} and \eqref{1229PKLIM}
to compute the stabilizing solution of coupled-AREs \eqref{GCARE1}.
Under the condition of Corollary \ref{1234stLQ},
the iterative procedure can be basically seen as an adaptation of the Kleinman algorithm \cite{Kleinman1968}
and also has a fast convergence rate.
\end{remark}

The result of Corollary \ref{1234stLQ} is consistent with Theorem 5.8 in \cite{Costa2015}.  It is noteworthy that, since $Q$ in the coupled-AREs \eqref{GCARE1} is sign-indefinite (i.e., $Q \in \mathcal{SH}^{n}_{\infty}$), we can further deduce the following corollary to Theorem 5, which provides condition for the existence of the stabilizing solution of coupled-AREs \eqref{GCARE1}
associated with the $H_{\infty}$ control problem.
\begin{corollary}\label{1237Hinfc}
Fix $Q\in\mathcal{H}_{\infty}^{n-}$ and $R\in{\mathcal{H}_{\infty}^{m+*}}$.
Suppose that there exists $P\in\mathcal{SH}_{\infty}^{n}$ such that $P(\ell)\leq-\xi I$ for some $\xi>0$
and satisfying the following inequality
\begin{equation}\label{352ineq}
\left[
  \begin{array}{cc}
    -P(\ell)+\mathcal{T}_{A}(P)(\ell)+Q(\ell)   & \mathcal{G}(P)(\ell)^{T} \\
     \mathcal{G}(P)(\ell) & \mathcal{R}(P)(\ell) \\
  \end{array}
\right]\gg 0\  \mu\text{-}a.e.,
\end{equation}
then coupled-AREs \eqref{GCARE1} admits a unique stabilizing solution  satisfying \eqref{SignCon771}.
\end{corollary}
\begin{proof}
In view of \eqref{352ineq}, it follows from Theorem 4 in \cite{Xiao2023} that
$(A|\mathbb{G})$ is EMSS, which implies that
$(A, G|\mathbb{G})$ is exponentially mean-square stabilizable.
Moreover,
according to the generalized Schur complement lemma,
there exists $\bar{P}\in \mathbf{S}_{+}$ satisfying $\mathcal{W}(\bar{P})(\ell)\gg 0$ and $\bar{P}(\ell)\leq-\xi_{0} I \  \mu\text{-}a.e.$ for some $\xi_{0}>0$.
Further, by Lemma 2,
there exists $\mathcal{W}(\bar{P})^{\frac{1}{2}}\in\mathcal{H}_{\infty}^{n+*}$ such that $\mathcal{W}(\bar{P})=[\mathcal{W}(\bar{P})^{\frac{1}{2}}]^{2}.$
Then $\mathcal{W}(\bar{P})^{-\frac{1}{2}}\in\mathcal{H}_{\infty}^{n+*}$ and it holds that $\mathcal{W}(\bar{P})^{-\frac{1}{2}}(\ell)\mathcal{W}(\bar{P})^{\frac{1}{2}}(\ell)=I\  \mu\text{-}a.e..$
For any $\ell\in\Theta$, let $H(\ell)=-[A(\ell)+G(\ell)\mathcal{M}(\bar{P})(\ell)]
\mathcal{W}(\bar{P})^{-\frac{1}{2}}(\ell).$
Note that
$\bar{P}(\ell)-[A(\ell)+G(\ell)\mathcal{M}(\bar{P})(\ell)+H(\ell)
\mathcal{W}(\bar{P})^{\frac{1}{2}}(\ell)]^{T}\mathcal{E}(\bar{P})(\ell)
[A(\ell)+G(\ell)\mathcal{M}(\bar{P})(\ell)+H(\ell)
\mathcal{W}(\bar{P})^{\frac{1}{2}}(\ell)]=\bar{P}(\ell)\leq -\xi_{0}I\ \mu\text{-}a.e..$
Applying Theorem 4 from \cite{Xiao2023} again, we have that there exists $\bar{P}\in{\mathbf{S}_+}$ such that $(A+G\mathcal{M}(\bar{P});\mathcal{W}(\bar{P})^{\frac{1}{2}}|\mathbb{G})$ is detectable.
The desired result is obtained by Theorem \ref{1160exmax}.
\end{proof}
\begin{remark}
At the end of this section, we would like to talk more about coupled-AREs \eqref{GCARE1}.
For finite MJLSs, when taking account of the LQ optimization problem
where the weighting matrices of the state and control are allowed to be indefinite (for example, see \cite{VALLECOSTA2008391}),
there exists a more general class of coupled-AREs:
\begin{equation}\label{weini717}
P(i)=\mathcal{T}_{A}(P)(i)+Q(i)
-\mathcal{G}(P)(i)^{T}\mathcal{R}(P)(i)^{+}
\mathcal{G}(P)(i)
\end{equation}
with $\mathcal{R}(P)(i)\geq 0$, $i\in\Theta$,
where $\mathcal{R}(P)(i)^{+}$ denotes the Moore-Penrose inverse.
Coupled-AREs \eqref{weini717} can be used to consider quite general problems (see \cite{VALLECOSTA2008391}).
However, there are some limitations in generalizing these results to MJLSs with the Markov chain on a Borel space.
In fact, if the sign condition \eqref{SignCon771} is weakened to
allow $\mathcal{R}(P)(\ell)>0\ \mu\text{-}a.e.$ or $\mathcal{R}(P)(\ell)\geq 0\ \mu\text{-}a.e.$,
boundedness and measurability of the considered matrix-valued function may not be guaranteed.
Consider an example with the Markov chain taking values in a countably infinite set $\mathbb{N}^{+}$.
For $\mathcal{R}(P)\in\mathcal{SH}_{\infty}^{2}$ given by
$$\mathcal{R}(P)(i)=
\left(
 \begin{array}{cc}
   \frac{1}{i} & 0\\
     0 & 0 \\
   \end{array}
\right),\
i\in\mathbb{N}^{+},$$
we have that $\mathcal{R}(P)^{+}$ is not in $\mathcal{SH}_{\infty}^{2}$
since $\|\mathcal{R}(P)^{+}\|_{\infty}=\sup\{i, i\in\mathbb{N}^{+}\}=\infty$.
\end{remark}

\section{The Infinite Horizon Control Design}\label{Gamecontrol}
\subsection{The Game-Based Control}\label{subsec3}
This subsection focuses on designing the game-based controller for MJLSs with the Markov chain taking values in $\Theta$.
Before that, we first address the infinite horizon LQ optimization problem associated with $(A,G|\mathbb{G})$.

Consider $(A,G|\mathbb{G})$.
Define the set of all admissible controls as
\begin{align*}
\begin{split}
&\mathbb{U}:=\{u|u\in{l^{2}(\mathbb{N};\mathbb{R}^{m})}\ \text{with}\ u(k),\ k\in\mathbb{N}\\
&\ \ \ \ \ \ \ \ \ \ \ \ \ \ \ \ \ \text{exponentially stabilizing}\ (A,G|\mathbb{G})\},
\end{split}
\end{align*}
where $l^{2}(\mathbb{N};\mathbb{R}^{m}):=\{u|u=\{u(k), k\in{\mathbb{N}}\}$ is an $\mathbb{R}^{m}$-valued random sequence
with $u(k)$ being $\mathfrak{F}_{k}$-measurable
and satisfying $\|u\|_{l^{2}(\mathbb{N};\mathbb{R}^{m})}=[\sum_{k=0}^{\infty}\mathbf{E}\|u(k)\|^{2}]^{\frac{1}{2}}<\infty\}$.
For $u\in\mathbb{U}$, define the cost function associated with $(A,G|\mathbb{G})$ as
\begin{equation}\label{infiJ867}
\begin{split}
&J\left(x_{0}, \vartheta_{0}; u\right):=\\
&\sum_{k=0}^{\infty}\mathbf{E}\left[x(k)^{T} Q\left(\vartheta(k)\right) x(k)+u(k)^{T}R\left(\vartheta(k)\right)u(k)\right],
\end{split}
\end{equation}
where $Q\in\mathcal{SH}^{n}_{\infty}$ and $R\in\mathcal{H}^{m+*}_{\infty}$.
The infinite horizon LQ optimization problem associated with $(A,G|\mathbb{G})$
 is to find a control strategy $u^{*}\in\mathbb{U}$
such that $J\left(x_{0}, \vartheta_{0}; u^{*}\right)\leq J\left(x_{0}, \vartheta_{0}; u\right)$ over $u\in\mathbb{U}$.
Define $V\left(x_{0},\vartheta_{0}\right):=\inf _{u \in \mathbb{U}} J\left(x_{0}, \vartheta_{0}; u\right).$
$u^*\in \mathbb{U}$ is called an optimal control law if it achieves the optimal cost value $V\left(x_{0},\vartheta_{0}\right)$.

The following lemma demonstrates that the optimal control law and the corresponding cost value can be expressed
in terms of the stabilizing solution of coupled-AREs \eqref{GCARE1}.
\begin{lemma}\label{inde899the}
Suppose that $(A, G|\mathbb{G})$ is exponentially mean-square stabilizable.
If the coupled-AREs \eqref{GCARE1} admit a stabilizing solution $P\in\mathcal{SH}_{\infty}^{n}$,
then the infinite horizon LQ optimization problem
associated with $(A,G|\mathbb{G})$
is solvable, and the optimal cost value is
$$V(x_{0},\vartheta_{0})
=x_{0}^{T}\mathbf{E}\{P(\vartheta_{0})\}x_{0}.$$
This optimal cost is achieved through the optimal control law
\begin{equation}\label{899opu}
u^{*}(k)=\mathcal{M}(P)(\vartheta(k))x(k),\ k\in\mathbb{N},
\end{equation}
and $x(k)$ is the system state of $(A,G|\mathbb{G})$.
\end{lemma}
\begin{proof}
For any $\tau\in\mathbb{N}$ and the stabilizing solution $P\in\mathcal{SH}_{\infty}^{n}$ of coupled-AREs \eqref{GCARE1},
it can be established that for each $k\in\overline{0, \tau-1},$
\begin{equation*}
\begin{split}
&\mathbf{E}\{x(k+1)^{T}P(\vartheta(k+1))x(k+1)
-x(k)^{T}P(\vartheta(k))x(k)\}\\
=&\mathbf{E}\{x(k)^{T}[\mathcal{T}_{A}(P)(\vartheta(k))-P(\vartheta(k))]x(k)\\
&\ \ \ +u(k)^{T}G(\vartheta(k))^{T}\mathcal{E}(P)(\vartheta(k))G(\vartheta(k))u(k)\\
&\ \ \ +2u(k)^{T}\mathcal{G}(P)(\vartheta(k))x(k)\}.
\end{split}
\end{equation*}
By taking the sum over $k\in\overline{0, \tau-1}$, and then letting $\tau\rightarrow+\infty$,
we can write \eqref{infiJ867} as
\begin{equation*}
\begin{split}
&J\left(x_{0}, \vartheta_{0}; u\right)=x_{0}^{T}\mathbf{E}\{P(\vartheta_{0})\}x_{0}\\
&+\sum_{k=0}^{\infty}\mathbf{E}\left[x(k)^{T}(\mathcal{T}_{A}(P)
(\vartheta(k))-P(\vartheta(k))+Q(\vartheta(k)))x(k)\right.\\
&\left.+u(k)^{T}\mathcal{R}(P)(\vartheta(k))u(k)+2u(k)^{T}
\mathcal{G}(P)A(\vartheta(k))x(k)\right].
\end{split}
\end{equation*}
Since $P$ satisfies \eqref{GCARE1}, by completing the square, we have that
\begin{equation*}
\begin{split}
J\left(x_{0}, \vartheta_{0}; u\right)&
\!=\!
\sum_{k=0}^{\infty}\mathbf{E}\left\{[u(k)\!-\!\mathcal{M}(P)
(\vartheta(k))x(k)]^{T}
\mathcal{R}(P)(\vartheta(k))\right.\\
\cdot&[u(k)-\mathcal{M}(P)(\vartheta(k))x(k)]\}
+x_{0}^{T}\mathbf{E}\{P(\vartheta_{0})\}x_{0}.
\end{split}
\end{equation*}
Consider $u^*=\{u^*(k),\ k\in\mathbb{N}\}$ with $u^{*}(k)$ given by \eqref{899opu}.
As $P$ is the stabilizing solution of the coupled-AREs \eqref{GCARE1}, it follows that $u^*\in\mathbb{U}$.
Imposing $u^*$ on $(A,G|\mathbb{G})$ yields that $V\left(x_{0},\vartheta_{0}\right)= x_{0}^{T}\mathbf{E}\{P(\vartheta_{0})\} x_{0}$.
\end{proof}

If restricting $Q=C^{T}C\in\mathcal{H}_{\infty}^{n+}$ and $R=D^{T}D=\mathcal{I}\in\mathcal{H}_{\infty}^{n+*}$,
a standard LQ optimization problem is introduced with the cost function:
\begin{equation*}
\begin{split}
&J_{S}(x_{0}, \vartheta_{0}; u):= \\
&\sum_{k=0}^{\infty}
\mathbf{E}\left[x(k)^{T}C(\vartheta(k))^{T}C(\vartheta(k))x(k)
+u(k)^{T}u(k)\right].
\end{split}
\end{equation*}

Combining Corollary \ref{1234stLQ} and Lemma \ref{inde899the}, one can get the following result.
\begin{corollary}\label{1727h2lq}
Suppose that $(A, G|\mathbb{G})$ is exponentially mean-square stabilizable and $(A;C|\mathbb{G})$ is detectable,
then the coupled-AREs \eqref{GCARE1} have a unique stabilizing solution $P\in\mathcal{H}_{\infty}^{n+}$.
Furthermore,
$$\inf _{u \in \mathbb{U}} J_{S}\left(x_{0}, \vartheta_{0}; u\right)=x_{0}^{T}\mathbf{E}\{P(\vartheta_{0})\}x_{0}$$
can be achieved by the optimal control law
$$u^{*}(k)=\mathcal{M}(P)(\vartheta(k))x(k),\ k\in\mathbb{N}.$$
\end{corollary}

Now we consider a two-player Nash game problem associated with dynamics \eqref{system}. In system \eqref{system},
$u(\cdot)\in l^{2}(\mathbb{N};\mathbb{R}^{m})$ and $v(\cdot)\in l^{2}(\mathbb{N};\mathbb{R}^{r})$
are viewed as the control inputs manipulated by two different players,
and given $\gamma>0$, the performance functions are defined as
\begin{align}
&J_1(x_{0}, \vartheta_{0}; u, v):=\sum_{k=0}^{\infty} \mathbf{E}\{\gamma^2\|v(k)\|^2-\|z(k)\|^2\},\nonumber\\
&J_2\left(x_{0},\vartheta_{0}; u, v\right):=\sum_{k=0}^{\infty} \mathbf{E}\{\|z(k)\|^2\}.\label{J2}
\end{align}
In this LQ nonzero-sum game problem,
Player 1 selects $v$ to minimize $J_1(x_{0}, \vartheta_{0}; u, v)$, and Player 2 selects $u$ to minimize $J_2(x_{0}, \vartheta_{0}; u, v)$.
A pair of strategies $(u, v)$ is said to be admissible for the infinite horizon two-player game if the system \eqref{system} in closed-loop with $(u, v)$ is EMSS.
The main objective of our work is to find Nash equilibrium strategies $(u^*, v^*)$ whose precise definitions are as follows.
\begin{definition}
A pair of admissible  strategies $(u^*, v^*)$,
 are called Nash equilibrium strategies
if for all admissible strategies $(u, v)$,
the following hold:
\begin{align}
&J_1\left(x_{0}, \vartheta_{0}; u^{*}, v^{*}\right) \leq J_1\left(x_{0}, \vartheta_{0}; u^{*}, v\right),\label{theworstdis1846}\\
&J_2\left(x_{0}, \vartheta_{0}; u^{*}, v^{*}\right) \leq J_2\left(x_{0},\vartheta_{0}; u, v^{*}\right).\label{theoptimalcon1848}
\end{align}
\end{definition}

The considered strategy space is given as follows:
\begin{align*}
&\Upsilon:=\{(u, v ) \in l^{2}(\mathbb{N};\mathbb{R}^{m}) \times l^{2}(\mathbb{N};\mathbb{R}^{r})| \\
& \ \ \ \ \ \ u(k)=K_{2}(\vartheta(k)) x(k),\ v(k)=K_{1}(\vartheta(k)) x(k),\ k\in\mathbb{N}\},
\end{align*}
where $K_{1}\in\mathcal{H}_{\infty}^{m\times n}$ and $K_{2}\in\mathcal{H}_{\infty}^{r\times n}$.
Specifically, both players are forced to employ linear, memoryless feedback strategies (see closed-loop no memory information structure in \cite{Basar1977IJGT}). It is essential to highlight that Nash equilibrium strategies are not necessarily linear (see \cite{Basar1974JOTA}). The reason behind selecting linear control strategies is two-fold: first, since the considered system is linear, the closed-loop system maintains its linear properties when linear control strategies are implemented; second, solving coupled-AREs to determine the linear feedback strategies is computationally more efficient. Additionally, enforcing the use of linear memoryless strategies avoids the difficulties associated with the non-uniqueness of Nash equilibrium strategies (see \cite{Basar1977IJGT}).

The following theorem provides a sufficient condition under which Nash equilibrium strategies exist for the LQ nonzero-sum game.

\begin{theorem}\label{gameth}
Assume that the following coupled-AREs:
\begin{equation} \label{areHinf}
\left\{
\begin{split}
&P_{1}(\ell)\!=\![A(\ell)\!+\!B(\ell)K_2(\ell)]^{T} \mathcal{E}(P_{1})(\ell)[A(\ell)\!+\!B(\ell)K_2(\ell)]\\
&\ \ \ \ \ \ \ \ \ \ -K_2(\ell)^{T}K_2(\ell) -C(\ell)^{T} C(\ell)\\
&\ \ \ \ \ \ \ \ \ \ -K_3(P_{1})(\ell)^{T} H_1(P_{1})(\ell)^{-1}K_3(P_{1})(\ell), \\
&H_1(P_{1})(\ell)=\gamma^{2} I+ \mathcal{T}_{F} (P_{1})(\ell) \gg 0, \ \ell\in\Theta,
\end{split}
\right.
\end{equation}
\begin{equation}\label{K1eq1477}
K_{1}(\ell)=-H_1(P_{1})(\ell)^{-1}K_3(P_{1})(\ell),\ \ \ \ \ \ \ \ \ \ \ \ \ \ \ \ \ \ \ \ \ \ \ \
\end{equation}
\begin{equation}\label{897ARE2}
\left\{
\begin{split}
&P_{2}(\ell)\!= \![A(\ell)\!+\!F(\ell)K_1(\ell) ]^{T} \mathcal{E}(P_{2})(\ell) [A(\ell)\!+\!F(\ell)K_1(\ell)]\\
&\ \ \ \ \ \!+\!C(\ell)^{T} C(\ell)\! -\!K_4(P_{2})(\ell)^{T}H_2(P_{2})(\ell)^{-1}K_4(P_{2})(\ell), \\
&H_2(P_{2})(\ell)=I+\mathcal{T}_{B}(P_{2})(\ell) \gg0,
\end{split}
\right.
\end{equation}
\begin{equation}\label{K2eq1491}
K_{2}(\ell)=-H_2(P_{2})(\ell)^{-1}K_4(P_{2})(\ell)\ \ \ \ \ \ \ \ \ \ \ \ \ \ \ \ \ \ \ \ \ \ \
\end{equation}
admit solutions $\left(P_1, K_{1}, P_2,K_{2}\right)$ with $P_{1}\in\mathcal{H}_{\infty}^{n-}$ and $ P_{2}\in\mathcal{H}_{\infty}^{n+}$,
where
\begin{eqnarray}
&K_{3}(P_{1})(\ell)=F(\ell)^{T} \mathcal{E}(P_{1})(\ell)[A(\ell)+B(\ell)K_2(\ell)],\label{equ2012K3}\\
&K_{4}(P_{2})(\ell)=B(\ell)^{T} \mathcal{E}(P_{2})(\ell)[A(\ell)+F(\ell)K_1(\ell)].\label{equ2014K4}
\end{eqnarray}
If $(A+FK_{1};C|\mathbb{G})$ is detectable, then the two-player LQ nonzero-sum game
has state feedback Nash equilibrium strategies $(u^{*},v^{*})\in \Upsilon$
with $u^{*}(k)=K_{2}(\vartheta(k))x(k)$ and $v^{*}(k)=K_{1}(\vartheta(k))x(k),$ $k\in\mathbb{N}$,
where $x(k)$ is the system state of \eqref{system}.
In this case,
$J_1\left(x_{0}, \vartheta_{0}; u^*, v^*\right)=x_{0}^{T} \mathbf{E}\{P_{1}(\vartheta_{0})\} x_{0}$
and
$J_2\left(x_{0}, \vartheta_{0}; u^*,v^*\right)=x_{0}^{T} \mathbf{E}\{P_{2}(\vartheta_{0})\} x_{0}.$
\end{theorem}
\begin{proof}
Construct $u^*(k)=K_2(\vartheta(k))x(k),\ k\in\mathbb{N}$ and apply it to \eqref{system}.
The resulting system and the corresponding performance function can be described as
\begin{equation}\label{a+bk21102}
\left\{\!
\begin{array}{ll}
x(k\!+\!1)\!=\![A(\!\vartheta(k))\!+\!B(\!\vartheta(k))
K_2(\!\vartheta(k))]x(k) \!+\!F(\!\vartheta(k))v(k), \\
z(k)=\left[
       \begin{array}{c}
         C(\vartheta(k))x(k) \\
         D(\vartheta(k))K_2(\vartheta(k))x(k) \\
       \end{array}
     \right],\ \forall k\in\mathbb{N},
\\
\end{array}
\right.
\end{equation}
and
\begin{equation*}
\begin{split}
&J_1\left(x_{0}, \vartheta_{0}; u^{*}, v\right)
=\sum_{k=0}^{\infty} \mathbf{E}\{\gamma^2\|v(k)\|^2\\
&-x(k)^{T}[C(\vartheta(k))^{T}C(\vartheta(k)) +K_2(\theta(k))^{T}K_2(\vartheta(k))]x(k)\}.
\end{split}
\end{equation*}
Notice that the detectability of $(A+FK_1;C|\mathbb{G})$ guarantees that $(A+B K_2+FK_1;\bar{C}|\mathbb{G})$ is detectable as well,
where
$\bar{C}(\ell)=[\begin{array}{cc}
                C(\ell)^{T} & K_2(\ell)^{T}
              \end{array}]^{T}$.
On the other hand, \eqref{897ARE2} can be rewritten as
\begin{equation*}
\begin{split}
P_{2}(\ell)&= [A(\ell)+B(\ell)K_2(\ell)+F(\ell)K_1(\ell) ]^{T} \mathcal{E}(P_{2})(\ell) \\
&\ \ \ \ \cdot[A(\ell)\!+\!B(\ell)K_2(\ell)\!+\!F(\ell)K_1(\ell)] \!+\!\bar{C}(\ell)^{T} \bar{C}(\ell)
\end{split}
\end{equation*}
with $P_{2}\in\mathcal{H}_{\infty}^{n+}$.
From Theorem \ref{Mthe1495if}, it follows that $(A+B K_2+FK_1|\mathbb{G})$ is EMSS.
Let $v^{*}(k)=K_{1}(\vartheta(k))x(k),\ k\in\mathbb{N}.$
Then, $(u^*, v^*) \in$ $l^{2}(\mathbb{N};\mathbb{R}^{m}) \times l^{2}(\mathbb{N};\mathbb{R}^{r})$.
Moreover, $v^*(k)$ exponentially stabilizes system \eqref{a+bk21102},
or equivalently, $P_{1}$ is the stabilizing solution of \eqref{areHinf}.
Thus, by Lemma \ref{inde899the}, one can conclude that $v^{*}(k)$ minimizes $J_1\left(x_{0}, \vartheta_{0}; u^{*}, v\right)$
and the corresponding optimal cost value is
$x_{0}^{T} \mathbf{E}\{P_{1}(\vartheta_{0})\} x_{0}=J_1\left(x_{0}, \vartheta_{0}; u^{*}, v^{*}\right)
\leq J_1\left(x_{0}, \vartheta_{0}; u^{*}, v\right).$

It remains to verify \eqref{theoptimalcon1848}.
To this end, similarly, consider
\begin{equation*}
\begin{split}
&J_2\left(x_{0},\vartheta_{0}; u, v^{*}\right)\\
=&\sum_{k=0}^{\infty} \mathbf{E}[x(k)^{T}C(\vartheta(k))^{T}C(\vartheta(k))x(k)+\|u(k)\|^2]
\end{split}
\end{equation*}
subject to
\begin{equation*}\label{1137a=fk1}
\left\{\!
\begin{array}{ll}
x(k\!+\!1)\!=\![A(\!\vartheta(k))\!+\!F(\!\vartheta(k))K_1(\!\vartheta(k))]x(k)
\!+\!B(\!\vartheta(k))u(k) ,\ \\
z(k)=\left[
       \begin{array}{c}
         C(\vartheta(k))x(k) \\
         D(\vartheta(k))u(k) \\
       \end{array}
     \right],\ \forall k\in\mathbb{N}.
\\
\end{array}
\right.
\end{equation*}
Since $(A+BK_2+FK_1|\mathbb{G})$ is EMSS, it means that $(A+FK_1,B|\mathbb{G})$ is exponentially mean-square stabilizable.
Additionally, in view of $(A+FK_1;C|\mathbb{G})$ being detectable,
from Corollary \ref{1727h2lq} we derive that $P_{2}$ is the stabilizing solution of coupled-AREs \eqref{897ARE2}.
Moreover, the optimal cost value is
$x_{0}^{T} \mathbf{E}\{P_{2}(\vartheta_{0})\} x_{0}=J_2\left(x_{0}, \vartheta_{0}; u^{*}, v^{*}\right)
\leq J_2\left(x_{0}, \vartheta_{0}; u, v^{*}\right)$
with $u^{*}(k)=K_{2}(\vartheta(k))x(k)$.
Therefore, it is claimed that the two-player LQ nonzero-sum game
has state feedback Nash equilibrium strategies $(u^{*},v^{*})\in \Upsilon$.
\end{proof}

\subsection{Application to the Mixed $H_2 / H_{\infty}$ Control}\label{subsec4}
Based on the aforementioned two-player Nash game,
this subsection deals with the infinite horizon $H_2 / H_{\infty}$ control problem.
To ensure that the associated the $H_{\infty}$ norm is well-defined
and the BRL can be used properly,
we assume that $\nu_{0}(\ell)>0$ $\mu\text{-}a.e.$ and $\int_{\Theta}g(t,\ell)\mu(dt)>0$ $\mu\text{-}a.e.$ in the remainder.

\begin{definition}\label{DefH2Hinf2069}
Given a disturbance attenuation level $\gamma>0$, find a control strategy $u^* \in l^{2}(\mathbb{N};\mathbb{R}^{m})$ such that

$(i)$~$u^*(k)$ exponentially stabilizes system  \eqref{system} internally,
i.e. when $v(k) \equiv0$ and $u(k)=u^*(k),\ k\in\mathbb{N}$, system \eqref{system} is EMSS;

$(ii)$~$\|\mathcal{L}_{u^*}\|<\gamma $ with
\small
\begin{equation*}
\begin{split}&\|\mathcal{L}_{u^*}\|\\
=&\esssup _{\substack{v \in l^{2}(\mathbb{N};\mathbb{R}^{r}),
 \\ v \neq 0, \ell\in{\Theta},\\  x_{0}=0 }}
 \frac{\left\{\sum_{k=0}^{\infty} \mathbf{E}\!\left[\|C(\vartheta(k)) x(k)\|^2\!+\!\|u^*(k)\|^2|\vartheta_0\!=\!\ell\right]\right\}^{\frac{1}{2}}}
 {\left\{\sum_{k=0}^{\infty} \mathbf{E}\left[\|v(k)\|^2|\vartheta_0=\ell\right]\right\}^{\frac{1}{2}}};
\end{split}
\end{equation*}

$(iii)$~When the worst-case disturbance (see \cite{Limebeer1994h2hinfNash}) $v^* \in l^{2}(\mathbb{N};\mathbb{R}^{r})$ is implemented,
where $v^{*}=
\mathop{argmin}\limits_{v}\{J_{1}(x_{0},\vartheta_{0};u^{*},v)\}$,
$u^*$ minimizes the performance function $J_2\left(x_{0},\vartheta_{0}; u, v\right)$ given by \eqref{J2}.\\
If there exist $\left(u^*, v^*\right)\in l^{2}(\mathbb{N};\mathbb{R}^{m}) \times l^{2}(\mathbb{N};\mathbb{R}^{r})$ that satisfy $(i)$-$(iii)$,
then we call that the infinite horizon $H_2 / H_{\infty}$ control problem is solvable and has a pair of solutions $\left(u^*, v^*\right)$.
%\deleted{Particularly, if there exists $u^*\in l^{2}(\mathbb{N};\mathbb{R}^{m})$ that satisfies $(i)$ and $(ii)$,
%then we call that the infinite horizon $H_{\infty}$ control problem is solvable and has a solution $u^*$.}
\end{definition}
\begin{definition}
Given a disturbance attenuation level $\gamma>0$, if there exists $u^*\in l^{2}(\mathbb{N};\mathbb{R}^{m})$ that satisfies only $(i)$ and $(ii)$ in Definition \ref{DefH2Hinf2069} but not necessarily satisfies $(iii)$,
then we call that the infinite horizon $H_{\infty}$ control problem is solvable and has a solution $u^*$.
\end{definition}

The following result can be directly achieved by applying the BRL established in \cite{Xiao2023} to system \eqref{a+bk21102}.
\begin{lemma}\label{BRL2198}
Given $\gamma>0$, for any $K_{2}\in\mathcal{H}^{r\times n}_{\infty}$, the following assertions are equivalent:

$(i)$~$u^{*}(k)=K_{2}(\vartheta(k))x(k),\  k\in\mathbb{N}$ exponentially stabilizes system  \eqref{system} internally,
and $\|\mathcal{L}_{u^*}\|<\gamma$;

$(ii)$~The coupled-AREs \eqref{areHinf} admit the stabilizing solution $P_{1}\in\mathcal{H}_{\infty}^{n-}$,
where $K_{3}(P_{1})$ is given by \eqref{equ2012K3}.
\end{lemma}

The existence conditions of an infinite horizon $H_2 / H_{\infty}$ controller are provided as follows:
\begin{theorem}\label{H2Hinfty2216}
($H_2 / H_{\infty}$ Control)
The following assertions hold:

$(i)$~Assume that coupled-AREs \eqref{areHinf}-\eqref{K2eq1491} admit solutions $\left(P_1, K_{1}, P_2,K_{2}\right)$
with $P_{1}\in\mathcal{H}_{\infty}^{n-}$ and $ P_{2}\in\mathcal{H}_{\infty}^{n+}$,
where $K_3(P_{1})$ and $K_{4}(P_{2})$ are given by \eqref{equ2012K3} and \eqref{equ2014K4}, respectively.
If $(A+FK_{1};C|\mathbb{G})$ is detectable,
then the infinite horizon  $H_2/H_{\infty}$ control problem is solved by $(u^{*},v^{*})$
with $u^{*}(k)=K_{2}(\vartheta(k))x(k)$ and $v^{*}(k)=K_{1}(\vartheta(k))x(k),$ $k\in\mathbb{N}$,
where $x(k)$ is the system state of \eqref{system};

$(ii)$~Assume that $(A+FK_{1};C|\mathbb{G})$ is detectable and
the infinite horizon  $H_2/H_{\infty}$ control problem is solved by $(u^{*},v^{*})$
with $u^{*}(k)=K_{2}(\vartheta(k))x(k)$ and $v^{*}(k)=K_{1}(\vartheta(k))x(k),$ $k\in\mathbb{N}$,
where $x(k)$ is the system state of \eqref{system},
then coupled-AREs \eqref{areHinf}-\eqref{K2eq1491} admit solutions $\left(P_1, P_2\right)$
with $P_{1}\in\mathcal{H}_{\infty}^{n-}$ and $ P_{2}\in\mathcal{H}_{\infty}^{n+}$.
\end{theorem}
\begin{proof}
$(i)$~By Theorem \ref{gameth}, it can be concluded that $(u^{*},v^{*})$ satisfies \eqref{theworstdis1846} and \eqref{theoptimalcon1848},
and specifically $v^{*}$ serves as the worst-case disturbance and $u^*$ minimizes \eqref{J2}.
Moreover, from the proof of Theorem \ref{gameth},
$P_{1}$ is confirmed to be the stabilizing solution of \eqref{areHinf}.
Applying Lemma \ref{BRL2198} yields that $u^{*}(k)=K_{2}(\vartheta(k))x(k)$ meets $(i)$ and $(ii)$ of Definition \ref{DefH2Hinf2069}.

$(ii)$~Since $(u^{*},v^{*})$ solves the  $H_2/H_{\infty}$ control problem,
from Lemma \ref{BRL2198} we have that coupled-AREs \eqref{areHinf} admit a stabilizing solution $P_{1}\in\mathcal{H}_{\infty}^{n-}$,
i.e., $(A+B K_2-F H_1(P_{1})^{-1}K_3(P_{1})|\mathbb{G})$ is EMSS.
Hence, $(A+BK_2,F|\mathbb{G})$ is exponentially mean-square stabilizable.
Applying Lemma \ref{inde899the},
it follows from $H_1(P_{1})(\ell) \gg 0$ that
$$v^{*}(k)=-H_1(P_{1})(\vartheta(k))^{-1}K_3(P_{1})(\vartheta(k))x(k),\ k\in\mathbb{N}$$
is the unique strategy to achieve the minimal value of $J_1\left(x_{0}, \vartheta_{0}; u^{*}, v\right)$.
This shows that \eqref{K1eq1477} holds $\mu$-$a.e.$.
Now $(A+B K_2+FK_1|\mathbb{G})$ is EMSS, i.e., $(A+FK_1,B|\mathbb{G})$ is exponentially mean-square stabilizable.
In addition, noting that $(A+FK_{1};C|\mathbb{G})$ is detectable,
from Corollary \ref{1727h2lq} we conclude that coupled-AREs \eqref{897ARE2} admit a unique stabilizing solution $P_{2}\in\mathcal{H}_{\infty}^{n+}$.
Moreover, the minimal value of $J_2\left(x_{0},\vartheta_{0}; u, v^{*}\right)$ is achieved by the unique optimal control
$$u^{*}(k)=-H_2(P_{2})(\vartheta(k))^{-1}K_4(P_{2})(\vartheta(k))x(k),\ k\in\mathbb{N},$$
and thereby \eqref{K2eq1491} holds $\mu$-$a.e.$.
\end{proof}

The subsequent corollary, derived from Lemma \ref{BRL2198}, describes how to design an infinite horizon $H_{\infty}$ controller.
\begin{corollary}\label{Hinfty2196}
($H_{\infty}$ Control)
The following assertions hold:

$(i)$~Assume that coupled-AREs \eqref{areHinf}, \eqref{K1eq1477}, and
\begin{equation}\label{K22201}
K_{2}(\ell)=-H_2(-P_{1})(\ell)^{-1}K_4(-P_{1})(\ell)
\end{equation}
admit solutions $\left(P_1, K_{1}, K_{2}\right)$ with $P_{1}\in\mathcal{H}_{\infty}^{n-}$,
where $K_3(P_{1})$ is given by \eqref{equ2012K3},
$$K_{4}(-P_{1})(\ell)=B(\ell)^{T} \mathcal{E}(-P_{1})(\ell)[A(\ell)+F(\ell)K_1(\ell)],$$
and
$$H_2(-P_{1})(\ell)=I+\mathcal{T}_{B}(-P_{1})(\ell).$$
If $(A+BK_{2}+FK_{1}|\mathbb{G})$ is EMSS,
then the infinite horizon $H_{\infty}$ control problem is solved by $u^{\infty}(k)=K_{2}(\vartheta(k))x(k)$,
%\deleted{and the worst-case disturbance is $v^{\infty}(k)=K_{1}(\vartheta(k))x(k)$, }
$k\in\mathbb{N}$,
where $x(k)$ is the system state of \eqref{system};

$(ii)$~If the infinite horizon $H_{\infty}$ control problem is solved by $u^{\infty}(k)=K_{2}(\vartheta(k))x(k)$, $k\in\mathbb{N}$,
where $x(k)$ is the system state of \eqref{system},
then coupled-AREs \eqref{areHinf} admit the stabilizing solution $P_{1}\in\mathcal{H}_{\infty}^{n-}$.
\end{corollary}

\section{Numerical Examples}\label{anexample}
\begin{example}\label{1468Solarthermal}
Consider the solar thermal receiver (proposed in \cite{Costa2015}) modeled by
\begin{equation}\label{biaoliangequ}
\left\{
\begin{array}{ll}
x(k+1)=a(\vartheta(k))x(k)+b(\vartheta(k))u(k),\\ z(k)=c(\vartheta(k))x(k),\ k\in\mathbb{N},\\
\end{array}
\right.
\end{equation}
in which the Markov chain $\{\vartheta(k), k\in\mathbb{N}\}$ takes values in $\Theta$.
Here, $\Theta=\Theta_{1}\bigcup\Theta_{2}$ and the measure $\mu$ on $\mathcal{B}(\Theta)$ is given as $\mu(\{i\}\times M)=\mu_{L}(M)$,
where $\Theta_{i}=\{i\times [0,1]\}$, $i\in\{1,2\}$,
$M\in{\mathcal{B}([0,1])}$, and $\mu_{L}(\cdot)$ is the Lebesgue measure on $\mathcal{B}([0,1])$.
The stochastic kernel of the Markov chain is defined as
$\mathbb{G}(\ell, j\times M)=\mathbb{P}(\vartheta(k+1)\in j\times M|\vartheta(k)=\ell)=P_{ij}\cdot\mu_{L}(M)$ for $\ell\in{\Theta_{i}}$
and $M\in{\mathcal{B}([0,1])}$, where $i,j\in\{1,2\}$.
In \eqref{biaoliangequ}, for $\vartheta(k)=(i,t)$, $i=1\ (\text{or}\ i=2)$ represents the weather condition being sunny (or cloudy),
while the variable $t\in [0,1]$ represents the effect of the instantaneous solar radiation on the system parameters
under the sunny (or cloudy) weather condition.
In light of this, $a(\cdot)$ changes with the variation of the instantaneous solar radiation
under the weather condition being sunny (or cloudy),
and follows $a(\ell)=a_{11}+t(a_{12}-a_{11})$ for $\ell=(1,t)$ (or $a(\ell)=a_{21}+t(a_{22}-a_{21})$ for $\ell=(2,t)$).
$b(\ell)=b_{i}$ and $c(\ell)=c$ for $\ell\in{\Theta_{i}}$, $i\in\{1,2\}$.
Set $a_{11}=0.93,\ a_{12}=0.73,\ a_{21}=0.94,\ a_{22}=1.1,$ $b_{1}=0.0915,\ b_{2}=0.0982,\ c=0.1885,$ $P_{11}=0.9767$, and $P_{22}=0.7611$.
It can be verified that $v(i,t)=[a(i,t)]^{2}$ and $k(i,t)=-t$ solve
\begin{equation*}
\begin{split}
&v(i,t)-[a(i,t)+b(i,t)k(i,t)]^{T}
\left[\sum_{j=1}^{2}\int^{1}_{0} P_{ij}v(j,s)ds\right]\\
&\ \ \ \ \ \ \ \ \ \cdot [a(i,t)+b(i,t)k(i,t)]\gg 0,
\end{split}
\end{equation*}
and $u(i,t)=[a(i,t)]^{2}$ and $h(i,t)=-t$ solve
\begin{equation*}
\begin{split}
&u(i,t)-[a(i,t)+h(i,t)c(i,t)]^{T}\left[\sum_{j=1}^{2}\int^{1}_{0} P_{ij}u(j,s)ds\right]\\
&\ \ \ \ \ \ \ \ \ \cdot [a(i,t)+h(i,t)c(i,t)]\gg 0
\end{split}
\end{equation*}
for any  $i\in\{1,2\}$ and $t\in[0,1]$.
By Theorem \ref{EMSSCiff}, we deduce that $(a,b|\mathbb{G})$ is exponentially mean-square stabilizable and $(a;c|\mathbb{G})$ is detectable.
\end{example}
\begin{example}\label{1468comparison}
In this example, consider a two-dimensional MJLS \eqref{system} with the Markov chain $\{\vartheta(k), k\in\mathbb{N}\}$
on the Borel space $(\Theta, \mathcal{B}(\Theta))$.
Here, $\Theta=\Theta_{1}\bigcup\Theta_{2}$, where $\Theta_{i}=\{i\times [0,1]\}$ for $i\in\{1,2\}$.
The measure $\mu$ on $\mathcal{B}(\Theta)$ is defined as $\mu(\{i\}\times M)=\mu_{L}(M)$,
$M\in\mathcal{B}([0,1])$.
The stochastic kernel of the Markov chain is given by
$\mathbb{G}(\ell, j\times M)=\mathbb{P}(\vartheta(k+1)\in j\times M|\vartheta(k)=\ell)=p_{ij}\cdot\mu_{L}(M)$ for $\ell\in{\Theta_{i}}$ and $M\in{\mathcal{B}([0,1])}$, where $i,j\in\{1,2\}$.
For any $\ell=(i,t)\in\Theta_{i}$ with $i\in\{1,2\}$ and $t\in[0,1]$,
let $A(\ell)=A(i,t)=A_{i1}+t(A_{i2}-A_{i1})$, $B(\ell)=B(i,t)=B_{i1}+t(B_{i2}-B_{i1})$, $C(\ell)=C(i,t)=tC_{i}$,
$F(\ell)=F(i,t)=tF_{i}$, and $D(\ell)=D(i,t)=D_{i}$.
The coefficients can be found in Table \ref{tabSyCoe}.
\begin{table}[h!t]
\setlength{\tabcolsep}{3pt}
   \centering
   \caption{Coefficients in Example \ref{1468Solarthermal}}\label{tabSyCoe}
\begin{tabular}{c c  c c c}\toprule[1pt]
Coefficients & $i=1$ & $i=2$
\\ \toprule[1pt]
$A_{i1}$&
$\left[
   \begin{array}{cc}
      2 & -1 \\
      0 & 0 \\
      \end{array}
\right]$
& $\left[
   \begin{array}{cc}
      0 & 1 \\
     0 & 2\\
      \end{array}
\right]$
\\
$A_{i2}$&
$\frac{1}{2}\times\left[
   \begin{array}{cc}
      2 & -1 \\
      0 & 0 \\
      \end{array}
\right]$
& $\frac{1}{2}\times\left[
   \begin{array}{cc}
      0 & 1 \\
     0 & 2\\
      \end{array}
\right]$
\\
$B_{i1}$&
$\left[
   \begin{array}{c}
      -0.6  \\
      -0.1 \\
      \end{array}
\right]$
& $\left[
   \begin{array}{c}
      -0.8 \\
     -1\\
      \end{array}
\right]$
\\
$B_{i2}$&
$\left[
   \begin{array}{c}
      1.3  \\
      0.7 \\
      \end{array}
\right]$
& $\left[
   \begin{array}{c}
      0.3 \\
     1.6\\
      \end{array}
\right]$
\\
$F_{i}$&
$\left[
   \begin{array}{c}
      0.4  \\
      -0.2 \\
      \end{array}
\right]$
& $\left[
   \begin{array}{c}
      -0.9 \\
     0.3\\
      \end{array}
\right]$
\\
$C_{i}$ &
$\left[
   \begin{array}{cc}
      -0.3 & -0.3 \\
      \end{array}
\right]$
&
$\left[
   \begin{array}{cc}
      0.4 &0.2\\
      \end{array}
\right]$
\\
$D_{i}$&
1
& 1
\\
$p_{ii}$&
0.15
& 0.1
\\
\hline
\end{tabular}
\end{table}

Set the disturbance attenuation level as $\gamma=0.5$.
We now deal with the infinite horizon $H_{2}/H_{\infty}$ control problem.
It should be noted that,
to obtain the numerical solutions of the integral coupled Riccati equations \eqref{areHinf}-\eqref{K2eq1491},
it is necessary to discretize the continuous term in $\Theta$.
Here we consider a finite grid approximation for $\Theta$.
The numerical solutions $\left(P_1, K_{1}, P_2,K_{2}\right)$ can be obtained through a backward iterative algorithm
and are visualized in Fig. \ref{picP1}-Fig. \ref{picK2}.
The algorithm adopted, similar to those in \cite{Hou2013JGO} and \cite{Zhangwh2008Auto},
involves iteratively computing the solutions of recursive equations:
\begin{equation*}
\begin{split}
&K_{3}(k)=F^{T}\mathcal{E}(P_{1}(k+1))(A+BK_{2}(k+1)),\\
&K_{4}(k)=B^{T}\mathcal{E}(P_{2}(k+1))(A+FK_{1}(k+1)),\\
&H_{1}(k+1)=\gamma^{2} \mathcal{I}+\mathcal{T}_{F}(P_{1}(k+1)),\\
&H_{2}(k+1)=\mathcal{I}+\mathcal{T}_{B}(P_{2}(k+1)),\\
&K_{1}(k)=-H_{1}(k+1)^{-1}K_{3}(k),\\
&K_{2}(k)=-H_{2}(k+1)^{-1}K_{4}(k),\\
&P_{1}(k)=\mathcal{T}_{A +BK_{2}(k)}(P_{1}(k+1))
-K_2(k)^{T}K_2(k)\\
&\ \ \ \ \ \ \ \ \ \ -C^{T} C-K_3(k)^{T} H_1(k+1)^{-1}K_3(k),\\
&P_{2}(k)=\mathcal{T}_{A +FK_{1}(k)}(P_{2}(k+1))
+C^{T} C\\
&\ \ \ \ \ \ \ \ \ \ -K_4(k)^{T} H_2(k+1)^{-1}K_4(k)
\end{split}
\end{equation*}
with the terminal value conditions
$P_{1}^{\tau}(\tau+1)=0$, $P_{2}^{\tau}(\tau+1)=0$,
$K_{1}^{\tau}(\tau+1)=0$, and $K_{2}^{\tau}(\tau+1)=0$, where $\tau\in\mathbb{N}$.
It is found that there exists a feasible solution
$U(i,t)=U_i$
for the inequality
\begin{equation*}
\begin{split}
U(i,t)-[&A(i,t)+F(i,t)K_{1}(i,t)]^{T}\left[\sum_{j=1}^{2}\int^{1}_{0} p_{ij}U(j,s)ds\right]\\
& \cdot [A(i,t)+F(i,t)K_{1}(i,t)]- I\geq 0,
\end{split}
\end{equation*}
where $i\in\{1,2\}$ and $t\in[0,1]$.
Here,
$$U_{1}=10^{3}\times\left[
                \begin{array}{cc}
                    1.3438  & -0.6177\\
                   -0.6177  &  0.4501\\
                \end{array}
              \right],$$
and
$$U_{2}=10^{3}\times\left[
                \begin{array}{cc}
                 0.1104  & -0.0044\\
                 -0.0044 &   1.3873\\
                \end{array}
              \right].$$
Hence, we conclude that $(A+FK_{1};C|\mathbb{G})$ is detectable.
On the other hand, it can be tested that $U(i,t)=U_i$ is also a feasible solution to the inequality
\begin{equation*}
\begin{split}
U(i,t)-&[A(i,t)+B(i,t)K_{2}(i,t)]^{T}
\left[\sum_{j=1}^{2}\int^{1}_{0} p_{ij}U(j,s)ds\right]\\
&\cdot[A(i,t)+B(i,t)K_{2}(i,t)]-I\geq 0,
\end{split}
\end{equation*}
which yields that $(A+BK_{2}|\mathbb{G})$ is EMSS.
Let $\vartheta_{0}$ be a random variable that follows
the distribution $\mu_{0}(\vartheta_{0}\in i\times M)=0.5\mu_{L}(M)$
for any $M\in\mathcal{B}([0,1])$ with $i\in\{1,2\}$.
In Fig. \ref{A+BK2_x1x2}, some trajectories of the system state of $(A+BK_{2}|\mathbb{G})$ with the initial conditions $(x_{0},\vartheta_{0})$ are plotted, in which the expected values ($\mathbf{E}[x_{h}(k)]=\frac{1}{1000}\sum_{s=1}^{1000}x^{s}_{h}(k)$,
and $x^{s}_{h}$ is the $s$-th sample path trajectory of $x_{h}$, $h\in\{1,2\}$) are also marked and
$x_{0}=\left[
\begin{array}{cc}
  2 & -2 \\
  \end{array}
  \right]^{T}$.
\begin{figure}
  \centering
  \includegraphics[width=\columnwidth]{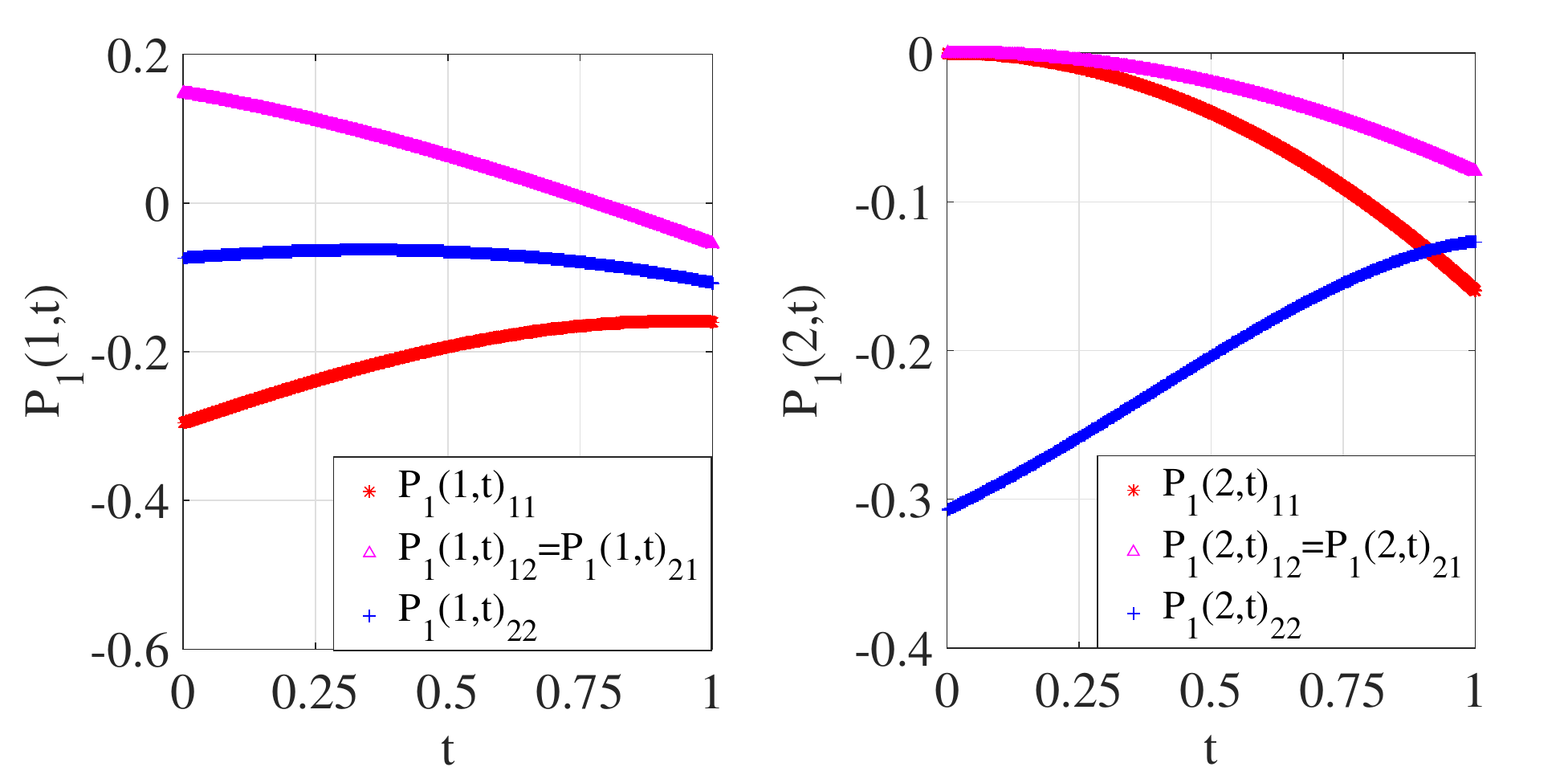}\\
  \caption{$P_1(1,t)$ $(P_1(2,t))$ is plotted on the left (right), and $P_1(i,t)_{hl}$ is the $(h,l)$-th entry of $P_1(i,t)$,
  $h, l\in\{1,2\}$.}\label{picP1}
\end{figure}
\begin{figure}
  \centering
  \includegraphics[width=\columnwidth]{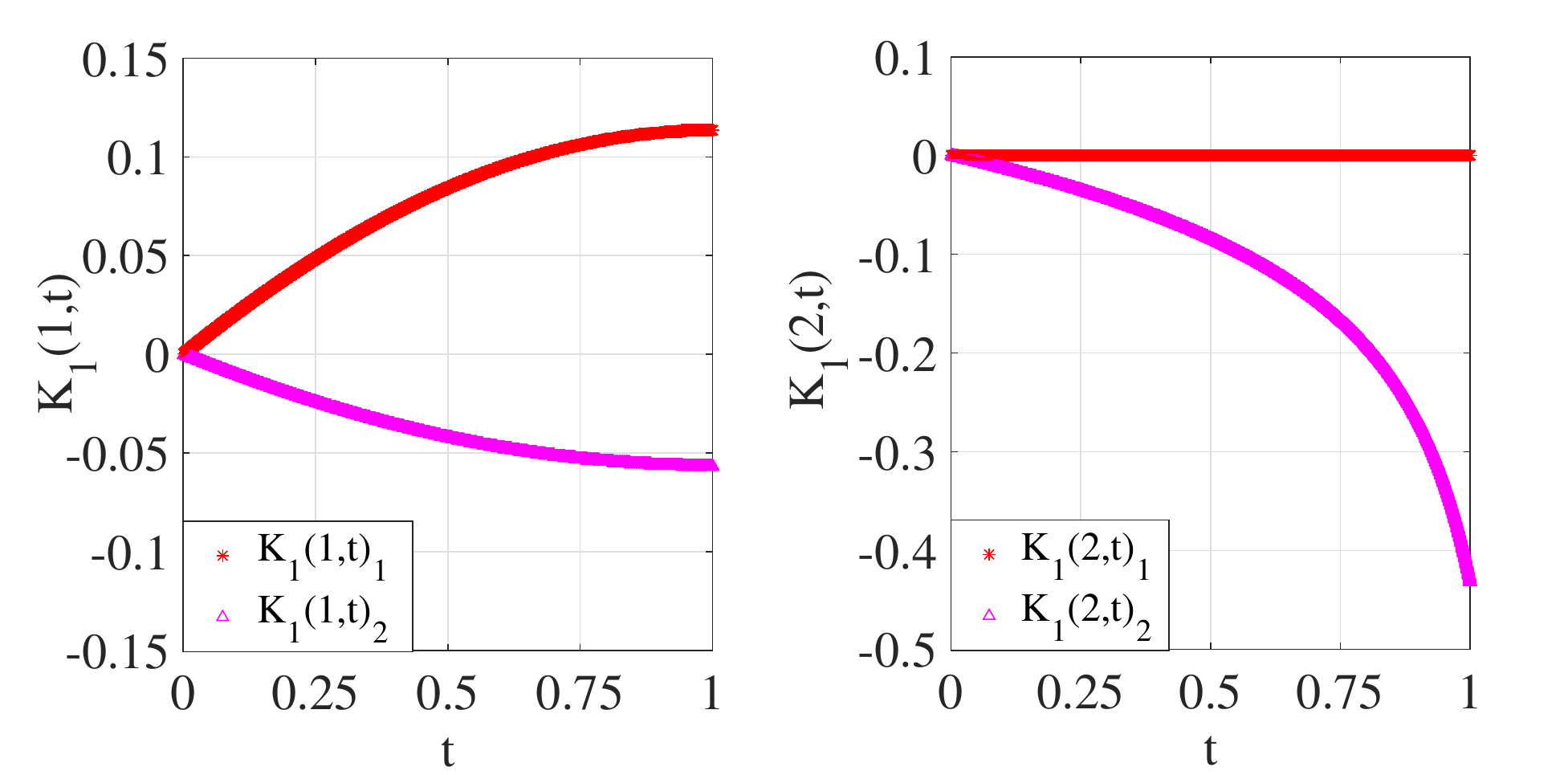}\\
  \caption {$K_1(1,t)$ $(K_1(2,t))$ is plotted on the left (right), and $K_1(i,t)_{h}$ is the $h$-th entry of $K_1(i,t)$, $h\in\{1,2\}$.}\label{picK1}
\end{figure}
\begin{figure}
  \centering
  \includegraphics[width=\columnwidth]{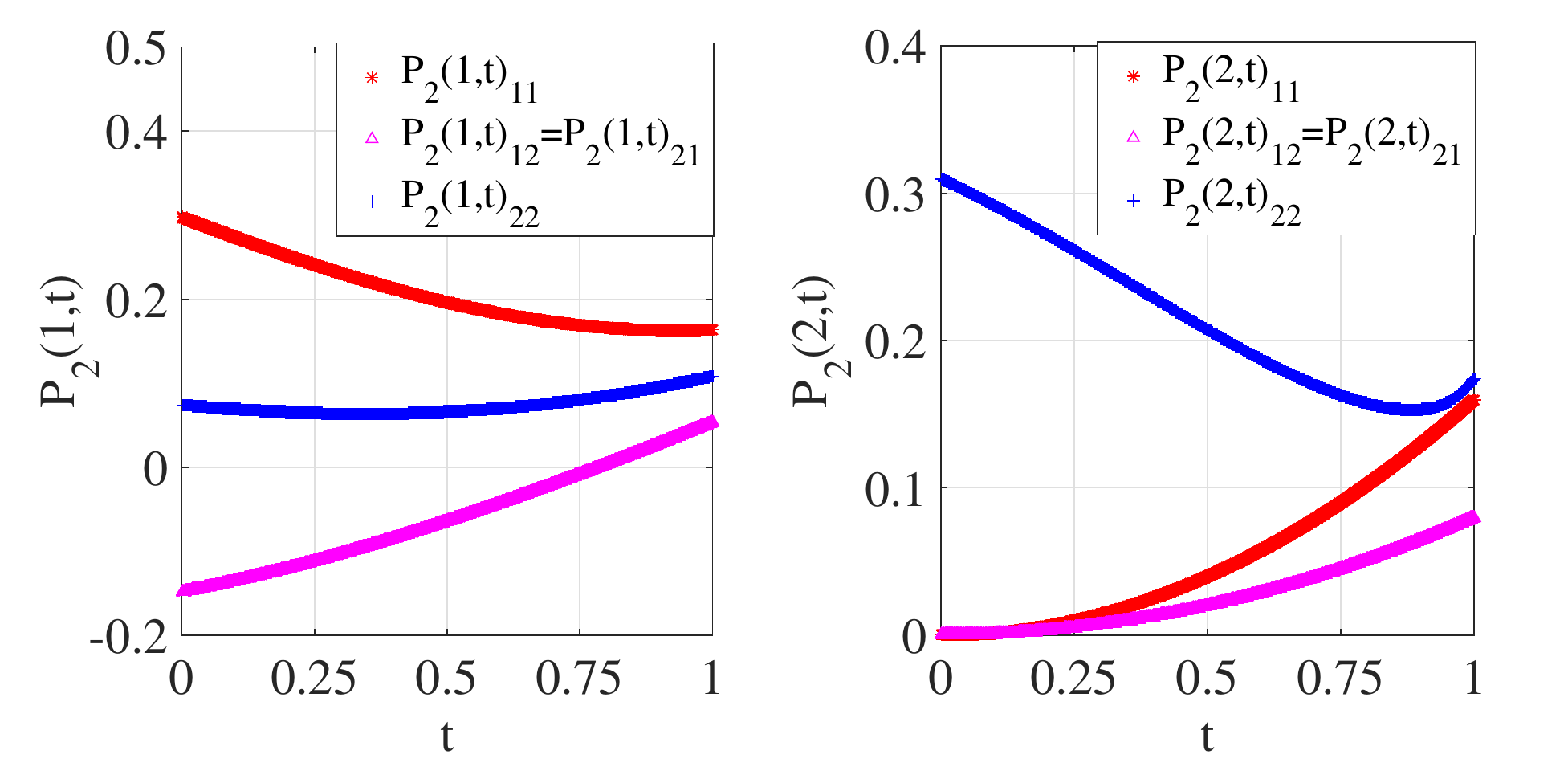}\\
  \caption{$P_2(1,t)$ $(P_2(2,t))$ is plotted on the left (right).}\label{picP2}
\end{figure}

\begin{figure}
  \centering
  \includegraphics[width=\columnwidth]{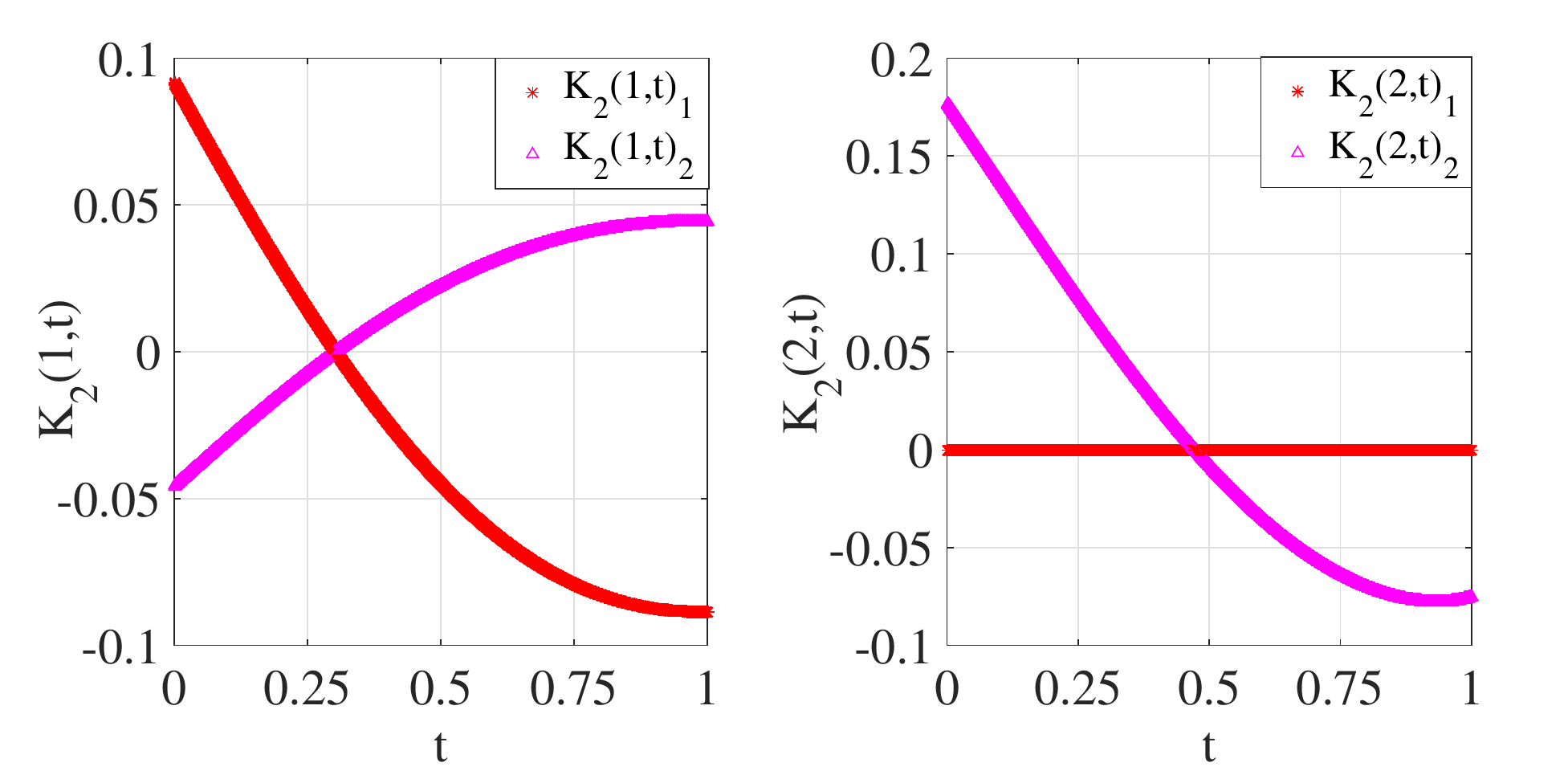}\\
  \caption{$K_2(1,t)$ $(K_2(2,t))$ is plotted on the left (right).}\label{picK2}
\end{figure}
\begin{figure}
  \centering
  \includegraphics[width=\columnwidth]{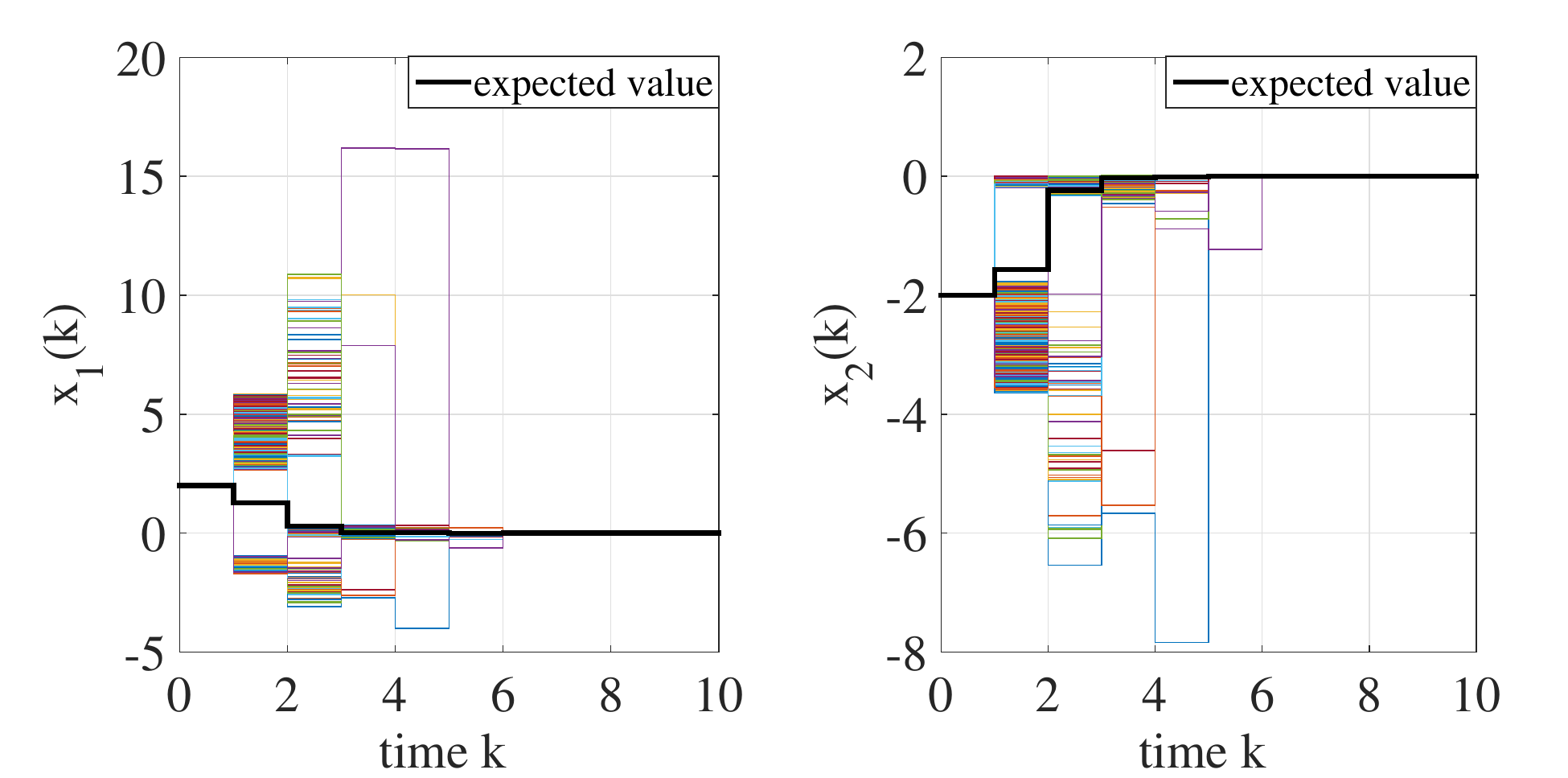}\\
  \caption{1000 possible trajectories of $x_{1}(k)$ $(x_{2}(k))$ for $(A+BK_{2}|\mathbb{G})$
  with the initial conditions $(x_{0},\vartheta_{0})$ are plotted on the left (right).}\label{A+BK2_x1x2}
\end{figure}

To solve the infinite horizon $H_{\infty}$ control problem with the attenuation level $\gamma$,
we consider the finite grid approximation of $\Theta$ adopted before and utilize a backward iterative algorithm like \cite{Hou2013JGO}.
This allows us to solve for the numerical solutions $\left(P_1^{\infty}, K_{1}^{\infty}, K_{2}^{\infty}\right)$
of coupled-AREs \eqref{areHinf}, \eqref{K1eq1477}, and \eqref{K22201}.
Moreover, it can be checked that $U(i,t)=U_i$ also satisfies \begin{equation*}
\begin{split}
U(i,t)-&[A+BK_{2}^{\infty}+FK_{1}^{\infty}]^{T}
\left[\sum_{j=1}^{2}\int^{1}_{0} p_{ij}U(j,s)ds\right]\\
&\cdot[A+BK_{2}^{\infty}+FK_{1}^{\infty}]- I\geq 0.
\end{split}
\end{equation*}
Therefore, $(A+BK_{2}^{\infty}+FK_{1}^{\infty}|\mathbb{G})$ is EMSS.
Fig. \ref{A+BK2inf+FK1inf_x1x2} displays some system state trajectories of $(A+BK_{2}^{\infty}+FK_{1}^{\infty}|\mathbb{G})$
under the initial conditions $(x_{0},\vartheta_{0})$.
\begin{figure}
  \centering
    \includegraphics[width=\columnwidth]{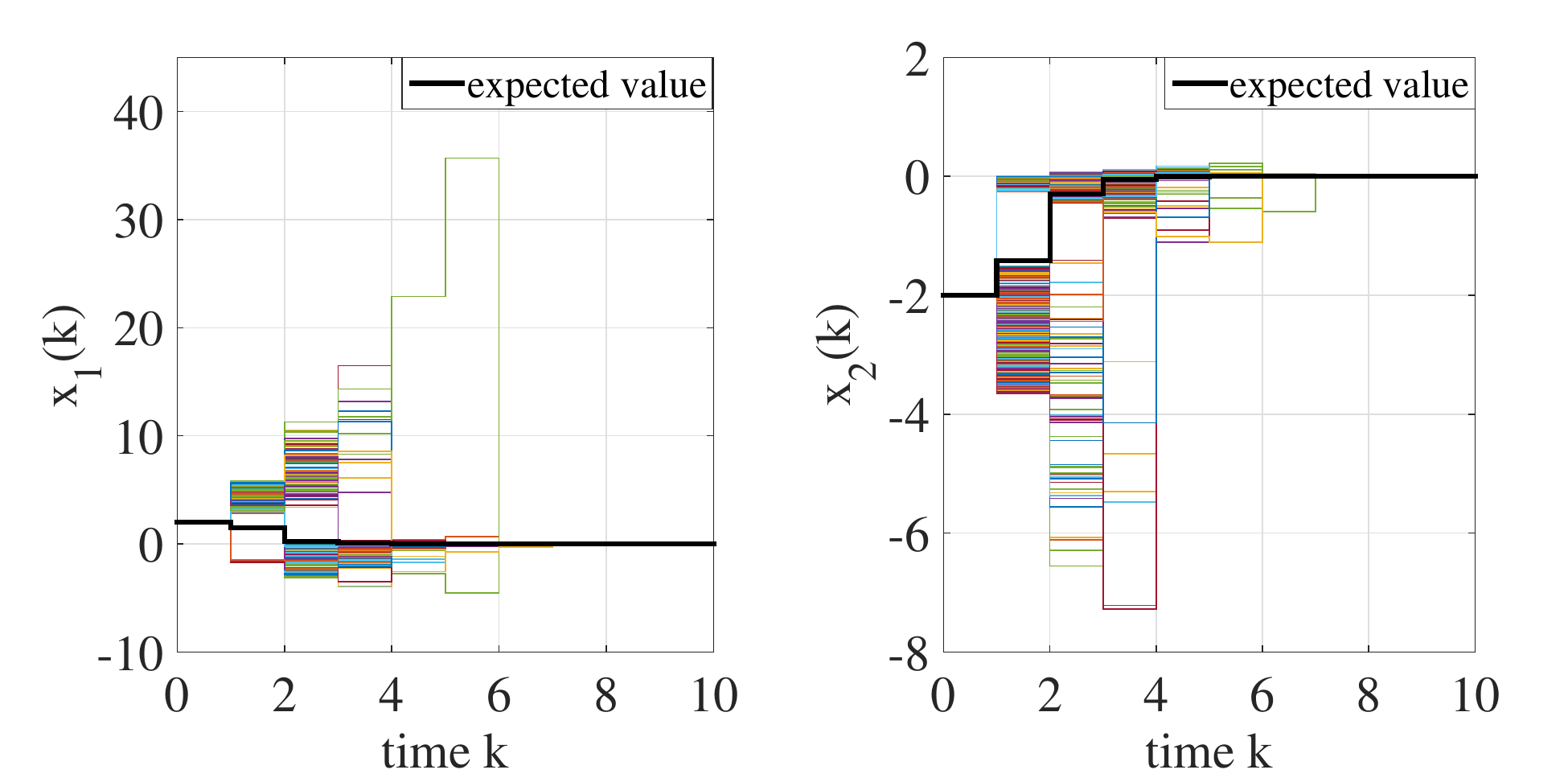}\\
  \caption{1000 possible trajectories of $x_{1}(k)$ $(x_{2}(k))$ for $(A+BK_{2}^{\infty}+FK_{1}^{\infty}|\mathbb{G})$
  with the initial conditions $(x_{0},\vartheta_{0})$ are plotted on the left (right).}\label{A+BK2inf+FK1inf_x1x2}
\end{figure}

Consider \eqref{system} with the disturbance $v(k)=e^{-2k}$, the control $u(k)=Kx(k)$, $k\in\mathbb{N}$,
and the initial conditions $(\bar{x}_{0},\vartheta_{0})$, where
$\bar{x}_{0}=\left[
\begin{array}{cc}
  0 & 0 \\
\end{array}
\right]^{T}$.
Evaluate the ratio $r_{K}(k)$ associated with \eqref{system}, which is defined as follows:
$$r_{K}(k)=\frac{[\sum_{d=0}^{k}\mathbf{E}\{|z(d)|^{2}\}]^{\frac{1}{2}}}{[\sum_{d=0}^{k}|v(d)|^{2}]^{\frac{1}{2}}},\ k\in\mathbb{N}.$$
The plots of $r_{K_2}(k)$ and $r_{K_2^{\infty}}(k)$ are shown in Fig. \ref{Outputv}.

Consider \eqref{system} with the initial conditions $(x_{0},\vartheta_{0})$.
Let $J_2(x_{0}, \vartheta_{0};u,v)(k)$ denote the performance function associated with \eqref{system}, given by
$$J_2(x_{0}, \vartheta_{0};u,v)(k)=\sum_{d=0}^{k} \mathbf{E}\{\|z(k)\|^2\}, \ k\in\mathbb{N}.$$
From Fig. \ref{J2campare}, it can be observed that when $k\rightarrow +\infty$,
$$J_{2}(x_{0}, \vartheta_{0};u^{*},v^{*})(k)\leq J_{2}(x_{0}, \vartheta_{0};u^{\infty},v^{*})(k),$$
where $u^{*}(k)=K_{2}x(k)$, $u^{\infty}(k)=K_{2}^{\infty}x(k)$, and $v^{*}(k)=K_{1}x(k),\  k\in\mathbb{N}.$
The comparison shows the superiority of $H_{2}/H_{\infty}$ control.
\begin{figure}
  \centering
    \includegraphics[width=\columnwidth]{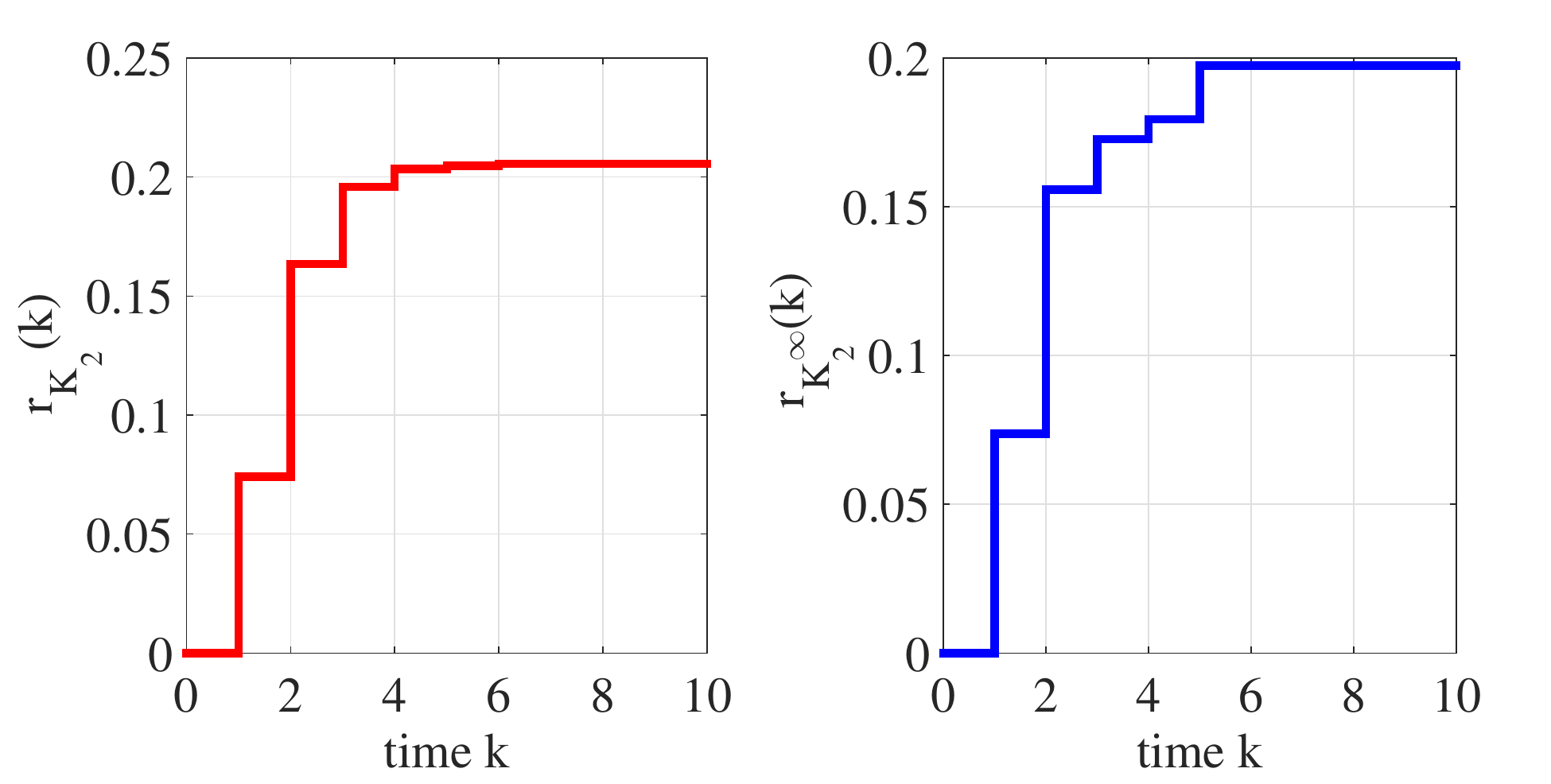}\\
  \caption{$r_{K_2}(k)$  ($r_{K_2^{\infty}}(k)$) is plotted on the left (right).}\label{Outputv}
\end{figure}
\begin{figure}
  \centering
    \includegraphics[width=\columnwidth]{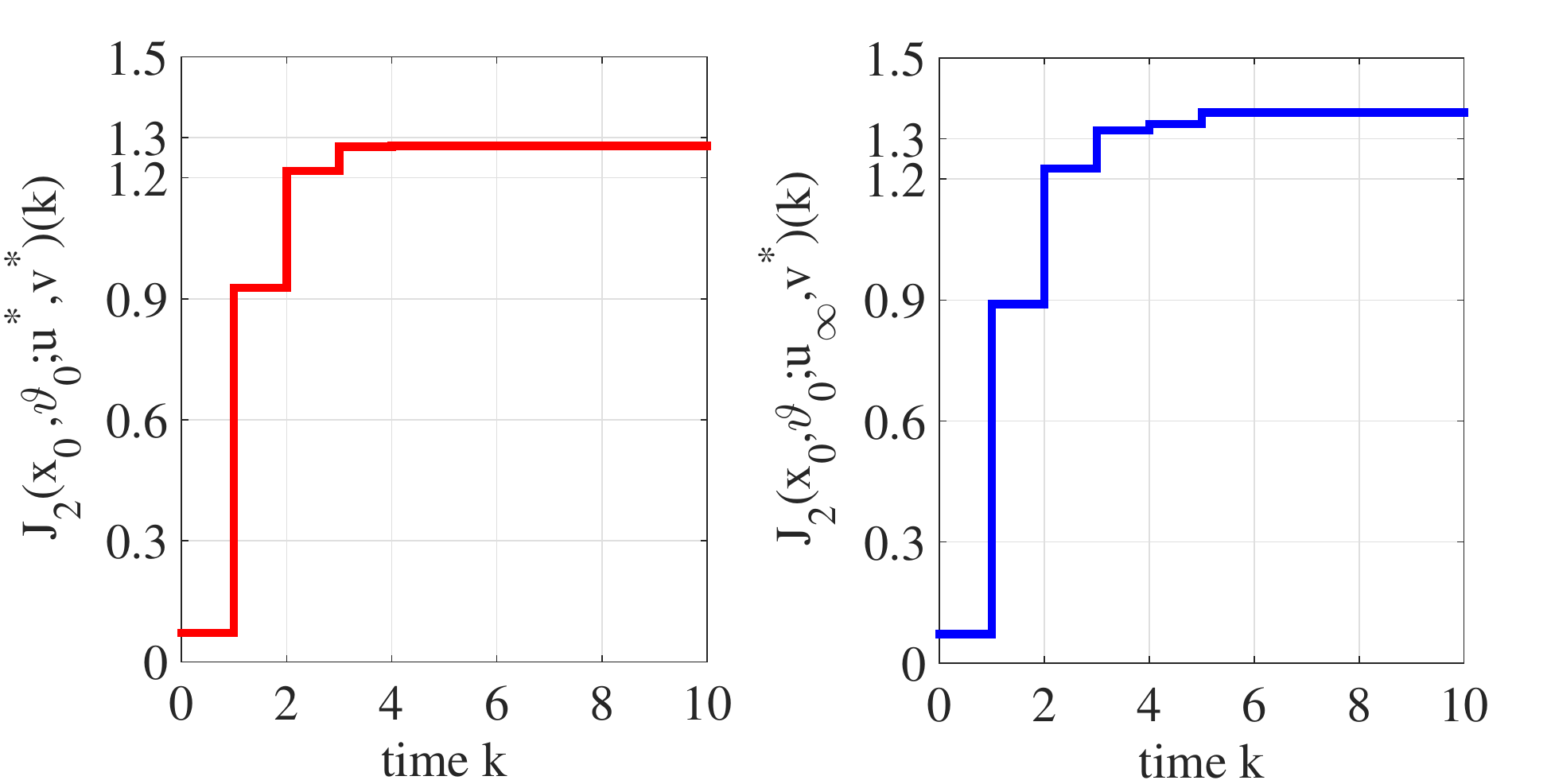}\\
  \caption{$J_{2}(x_{0}, \vartheta_{0};u^{*},v^{*})(k)$ ($J_{2}(x_{0}, \vartheta_{0};u_{\infty},v^{*})(k)$) is plotted on the left (right).}\label{J2campare}
\end{figure}
\end{example}

\section{Conclusion}\label{Conclusions}
This paper contains full investigations on detectability, coupled-AREs, and the game-based control
for MJLSs with the Markov chain on a Borel space.
It has been testified that under the detectability,
the autonomous system being EMSS depends on the existence of a positive semi-definite solution to a generalized Lyapunov equation.
This stability criterion has been used to determine the existence conditions of the maximal solution and the stabilizing solution to coupled-AREs.
Particularly, it has been shown that exponential mean-square stabilizability and detectability
can ensure the existence of the stabilizing solution to coupled-AREs related to the standard LQ optimization problem.
Further, by means of the stabilizing solution to four algebraic Riccati equations with the $H_{2}$-type and the $H_{\infty}$-type coupled,
the infinite horizon game-based control problem has been settled and Nash equilibrium strategies have been offered.
As an application, the mixed $H_{2}/H_{\infty}$ controller has been designed based on the BRL established in \cite{Xiao2023}.
These works have extended the previous results from the case where the Markov chain takes values in a countable set to the uncountable scenario.
In contrast to prior ones, it can be keenly aware that boundedness, measurability, and integrability
of the solutions to equations like the generalized Lyapunov equation and coupled-AREs,
bring significant challenges in analysis and the controller design of MJLSs with the Markov chain on a Borel space.
Finally, it should also be mentioned that under some mild technical assumptions,
such as that the Markov chain and the noise sequence are mutually independent,
the results established in this paper can be generalized to the system with multiplicative noise.
\section*{References}

\end{document}